\documentclass[12pt]{amsart}
\usepackage{amsfonts}
\usepackage{amssymb}
\usepackage{amsmath}
\input xy
\xyoption {all}

\usepackage{geometry}
\usepackage{amsfonts}
\usepackage{amssymb}
\usepackage{amsmath}
\usepackage{amsthm}
\usepackage{latexsym}
\usepackage{amscd}

\newtheorem{theorem}{Theorem}[section] 
\newtheorem{cor}[theorem]{Corollary}
\newtheorem{defn}[theorem]{Definition}
\newtheorem{example}[theorem]{Example}
\newtheorem{lemma}[theorem]{Lemma}
\newtheorem{prop}[theorem]{Proposition}
\newtheorem{remark}[theorem]{Remark}

\newcommand{\mh}{\mbox{MHM}}
\newcommand{\ms}{\mbox{mHs}}
\newcommand{\mc}{\mbox{MHC}_y}
\newcommand{\mt}{\mbox{MHT}}
\newcommand{\co}{\mbox{Coh}}

\newcommand{\LL}{{\mathcal L}}
\newcommand{\U}{{\mathcal U}}

\newcommand{\OO}{{\mathcal O}}
\newcommand{\HC}{{\mathcal H}}
\newcommand{\VV}{{\mathcal V}}

\newcommand{\FC}{{\mathcal F}}
\newcommand{\MC}{{\mathcal M}}
\newcommand{\PC}{{\mathcal P}}

\newcommand{\Z}{\mathbb{Z}}

\newcommand{\Q}{\mathbb{Q}}

\newcommand{\C}{\mathbb{C}}

\begin{document}

\title[Equivariant characteristic classes]{Equivariant characteristic classes\\ of singular complex algebraic varieties}

\author[S. E. Cappell ]{Sylvain E. Cappell}
\address{S. E. Cappell: Courant Institute, New York University, New York City, NY 10012}
\email {cappell@cims.nyu.edu}

\author[L. Maxim ]{Laurentiu Maxim}
\address{L. Maxim : Department of Mathematics, University of Wisconsin-Madison, 480 Lincoln Drive, Madison WI 53706-1388}
\email {maxim@math.wisc.edu}

\author[J. Sch\"urmann ]{J\"org Sch\"urmann}
\address{J.  Sch\"urmann : Mathematische Institut,
          Universit\"at M\"unster,
         Einsteinstr. 62, 48149 M\"unster,
          Germany.}
\email {jschuerm@math.uni-muenster.de}

\author[J. L. Shaneson ]{Julius L. Shaneson}
\address{J. L. Shaneson: Department of Mathematics, University of Pennsylvania, 209 S 33rd St, Philadelphia, PA-19104}
\email {shaneson@sas.upenn.edu}



\thanks{S. Cappell and J. Shaneson are partially supported by DARPA-25-74200-F6188. L. Maxim is partially supported by NSF-1005338. J. Sch\"urmann is supported by the
SFB 878 ``groups, geometry and actions". }

\date{\today}

\begin{abstract}
Homology Hirzebruch characteristic classes for singular varieties have been recently defined
by Brasselet-Sch\"urmann-Yokura as an attempt to unify previously known
characteristic class theories for {singular} spaces (e.g., MacPherson-Chern classes,
Baum-Fulton-MacPherson Todd classes, and Goresky-MacPherson $L$-classes, respectively). In this note we define equivariant analogues of these classes for {\em singular} quasi-projective varieties acted upon by a finite group of algebraic automorphisms,  and show how these can
be used to calculate the homology Hirzebruch classes of global quotient varieties. We also compute the new classes in the context of monodromy problems, e.g., for varieties that fiber equivariantly (in the complex topology) over a connected algebraic manifold. As another application, we discuss Atiyah-Meyer type formulae for twisted Hirzebruch classes of global orbifolds.
\end{abstract}

\maketitle

\tableofcontents

\section{Introduction}

Characteristic classes, generally defined for vector bundles, are very rich global cohomology
invariants measuring how far a local product structure is from being global. They are one of the
unifying concepts in algebraic topology, differential geometry and algebraic geometry. Characteristic classes of manifolds are defined via their tangent bundles, and provide a powerful tool in classification problems for manifolds (e.g., in surgery theory). In his seminal book \cite{H}, Hirzebruch provided a unifying theory of (cohomology) characteristic classes in the smooth context. More precisely, he defined a parametrized family of characteristic classes, $T_y$,
and showed that the well-known Chern, Todd, and respectively $L$-classes are just special cases of this, for $y=-1, 0, 1$, respectively.

Spaces with singularities on the other hand do not possess tangent bundles, their characteristic
classes being usually defined in homology. The problem of finding good notions of characteristic classes for singular spaces has been and still is the object of extensive research, e.g., see \cite{BFM,BFQ,CS,CSW,Che,GM1,MP,M}. In the complex algebraic context, the recently defined homology Hirzebruch classes of Brasselet, Sch\"urmann and Yokura \cite{BSY} have good functorial and normalization properties (e.g., for smooth spaces they are Poincar\'e dual to the usual cohomology Hirzebruch classes $T_y$, which appeared in the generalized Hirzebruch-Riemann-Roch theorem \cite{H}), and unify in a functorial sense the well-known Chern classes of MacPherson \cite{MP}, Todd classes of Baum, Fulton and MacPherson \cite{BFM}, and $L$-classes of Goresky-MacPherson \cite{GM1}, Cheeger \cite{Che} and Cappell-Shaneson \cite{CS1}, respectively.\\

One of the aims of this note is to construct an equivariant theory of homology Hirzebruch classes, namely the (Hodge-theoretic) {\em Atiyah-Singer classes} associated to any (possibly singular) quasi-projective variety $X$ acted upon by a finite group $G$ of algebraic automorphisms. The new classes, which for each $g \in G$  are denoted by $$ {T_y}_*(X;g) \in H_{ev}^{BM}(X^g)\otimes\C[y],$$ are supported on the fixed point sets of the action, and satisfy the  {\it normalization} property  asserting that if $X$ is non-singular  then  ${T_y}_*(X;g)$ is  Poincar\'e dual to the cohomology Atiyah-Singer class $T_y^*(X;g)$ defined implicitly in the statement of the holomorphic Lefschetz theorem (see \S\ref{ASm} for a definition of the latter):
\begin{equation}\label{intro} {T_y}_*(X;g)= T^*_y(X;g) \cap [X^g] \in H^{BM}_{ev}(X^g)\otimes\C[y]. \end{equation} 
Moreover, if $X$ is projective (but possibly singular), the degree of  the zero-dimensional component of ${T_y}_*(X;g)$ is the equivariant Hodge polynomial $\chi_y(X;g)$ studied in \cite{eCMS}, i.e., 
\begin{equation} \chi_y(X;g):=\sum_{i,p}(-1)^i {\rm trace} \left(g \vert {\rm Gr}^p_F H^i(X;\C) \right) \cdot (-y)^p =\int_{[X^g]} {T_y}_*(X;g),\end{equation} for $F^{\centerdot}$  the Hodge filtration of the canonical Deligne mixed Hodge structure on $H^*(X;\Q)$. If $X$ is not necessarily projective, it is more natural to consider the corresponding polynomial $\chi^c_y(X;g)$ defined similarly in terms of compactly supported cohomology. But even if  $X^g$ is projective, $\chi^c_y(X;g)$ does not necessarily agree with $\int_{[X^g]} {T_y}_*(X;g)$. In fact, if $\bar X$ is a (maybe singular) projective $G$-equivariant compactification of $X$, with $\partial X:={\bar X} \setminus X$ the part at infinity, then, by additivity, one gets for $X^g$ projective:
\begin{equation}\begin{split}
\chi^c_y(X;g)&=\chi_y({\bar X};g)-\chi_y(\partial X;g) \\ &= \int_{[X^g]} {T_y}_*(X;g)+ \int_{[{(\partial X)}^g]} [{T_y}_*(\bar X;g)-{T_y}_*(\partial X;g)] .
\end{split}\end{equation} So, if $\chi^c_y(X;g) \neq \int_{[X^g]} {T_y}_*(X;g)$, any $G$-equivariant compactification of $X$ must have $g$-fixed points ``at infinity".\\

The Hodge-theoretic Atiyah-Singer class ${T_y}_*(X;g)$ defined in this paper is an equivariant generalization of the homology Hirzebruch class ${T_y}_*(X)$ defined by Brasselet-Sch\"urmann-Yokura \cite{BSY} in the sense that 
\begin{equation}{T_y}_*(X;id)={T_y}_*(X) \otimes \C  \in H_{ev}^{BM}(X)\otimes\C[y].\end{equation}
In fact, the construction of our class ${T_y}_*(X;g)$ follows closely that of \cite{BSY}, and it comes in two flavors, a motivic one based on relative Grothendieck groups of $G$-equivariant varieties, and another one using $G$-equivariant mixed Hodge modules.

Let $K^G_0 (var/X)$ be  the relative Grothendieck group of $G$-equivariant quasi-projective varieties over $X$, i.e., the free abelian group of isomorphism classes $[Y \to X]$ of $G$-morphisms of such spaces, modulo the usual ``scissor" relation: 
$$[Y \to X]=[Z\to Y \to X]+[Y \setminus Z\to Y \to X],$$ for any $G$-invariant closed algebraic subspace $Z \subset Y$ (see also \cite{Bit}[Sect.7]). This group has the same functorial properties as in \cite{BSY}, e.g., push-forward (defined by composition of arrows), exterior product and open restrictions, as well as a forgetful functor $$K^G_0 (var/X) \to K^H_0 (var/X), $$ for any subgroup $H<G$. As in \cite{BSY,Sch3}, there is a natural transformation (as explained in the Appendix)
$$\chi^G_{\rm Hdg}: K^G_0 (var/X) \to K_0 (\mh^G(X))$$ to the Grothendieck group of $G$-equivariant mixed Hodge modules, mapping $[id_X]$ to the class of the constant  Hodge module $[\Q^H_X]$.

The {\em motivic Atiyah-Singer class transformation}
\begin{equation}
{T_y}_*(g):={\mt_y}_*(g) \circ \chi^G_{\rm Hdg} :K^G_0 (var/X) \to H_{ev}^{BM}(X^g)\otimes\C[y]
\end{equation}
can then be defined in terms of the {\em (Hodge-theoretic) Atiyah-Singer class transformation}
\begin{equation}
{\mt_y}_*(g) :K_0 (\mh^G(X)) \to H_{ev}^{BM}(X^g)\otimes\C[y^{\pm 1},(1+y)^{-1}]
\end{equation} on the Grothendieck group of $G$-equivariant mixed Hodge modules. The value taken by these transformations on $[id_X]$, resp. on  the (class of the) constant $G$-equivariant Hodge sheaf $\Q^H_X$ 
yields the Atiyah-Singer class ${T_y}_*(X;g)$ in  the even-degree (Borel-Moore) homology of the fixed-point set $X^g$, with complex polynomial coefficients. In the case when $g$ is the identity element of $G$, these transformations reduce to the complexification of their non-equivariant counterparts ${T_y}_*$ and ${\mt_y}_*$ from \cite{BSY,Sch3}.

One advantage of the mixed Hodge module approach is that we can  evaluate the transformation ${\mt_y}_*(g)$ on other interesting ``coefficients''. E.g.,  a ``good" $G$-equivariant variation $\LL$ of mixed Hodge structures on a smooth $X$ yields twisted Atiyah-Singer classes ${T_y}_*(X,\LL;g)$,  the $G$-equivariant intersection cohomology Hodge module $IC^H_X$ yields, for $X$ pure-dimensional, similar localized classes ${IT_y}_*(X;g)$, and  for a generically defined ``good" (i.e., graded polarizable, admissible and with quasi-unipotent monodromy at infinity) $G$-equivariant variation $\LL$ of mixed Hodge structures on $X$ we obtain twisted Atiyah-Singer classes ${IT_y}_*(X,\LL;g)$ associated to the $G$-equivariant twisted intersection Hodge module $IC^H_X(\LL)$. \\

The construction of these Atiyah-Singer class transformations also uses  Saito's theory of algebraic mixed Hodge modules  \cite{Sa} to first define an equivariant version of the motivic Chern class transformation of \cite{BSY,Sch3}, i.e., the {\em equivariant motivic Chern class transformation}:
\begin{equation}
\mc^G:K_0(\mh^G(X)) \to K_0({\rm Coh}^G(X)) \otimes \Z[y^{\pm 1}],
\end{equation}
for $K_0({\rm Coh}^G(X))$ the Grothendieck group of $G$-equivariant algebraic coherent sheaves on $X$. We want to emphasize here that our construction of $\mc^G$ relies heavily on our notion of {\em weak equivariant derived categories} (as defined in the Appendix) to adapt Saito's functors ${\rm gr}^F_pDR$ to this equivariant context. Our approach is much simpler than using the corresponding notion of equivariant derived categories of \cite{BL}, which a priori is not adapted to Saito's filtered de Rham functors. 

We next use (a suitable twisted version of) the Lefschetz-Riemann-Roch transformation of Baum-Fulton-Quart \cite{BFQ} and Moonen \cite{M}:
\begin{equation}
td_*(g)(-):K_0(\co^G(X)) \to H^{BM}_{ev}(X^g;\C)
\end{equation}
to obtain (localized) homology classes on the fixed-point set $X^g$.
In this way, many properties (and their proofs) of these transformations, e.g.,  functoriality under push-down for proper maps and restriction to open subsets as well as multiplicativity for exterior products,  can be obtained by formally adapting the corresponding ones from the non-equivariant context, as in \cite{BSY,Sch3}. While we are brief in describing such results, we emphasize much more the new results specifically related to the equivariant situation. \\

Over a point space, the transformation ${\mt_y}_*(g)$  coincides with the
equivariant $\chi_y(g)$-genus ring homomorphism $\chi_y(g):K^G_0(\ms^p) \to
\C[y,y^{-1}]$, defined on the Grothendieck group of the abelian category ${G-}\ms^p$ of $G$-equivariant (graded) polarizable mixed Hodge structures by 
\begin{equation} \chi_y(g)([H]):=\sum_p {\rm trace}(g \vert  Gr^p_F(H \otimes \C)) \cdot (-y)^p,
\end{equation}
for $F^{\centerdot}$ the Hodge filtration of $H \in {G-}\ms^p$. \\

One advantage of the motivic approach is that it unifies the Lefschetz-Riemann-Roch transformation $td_*(g)(-)$ (for $y=0$), as well as the equivariant Chern class transformation $c_*(g)(-):=c_*(tr_g(-|_{X^g}))$ (for $y=-1$) mentioned in \cite{Sch4}[Ex.1.3.2], which one gets by evaluating the MacPherson-Chern class transformation on the constructible function given by taking stalk-wise traces over $X^g$.  More precisely, there is a commutative diagram of transformations (see Remark \ref{normal2}):
\begin{equation}\label{comm}\begin{CD}
K_0(D^{b,G}_c(X;\Q)) @ <{rat} \circ \chi^G_{\rm Hdg} << K^G_0(var/X) @> {\rm MHC}^G_{y=0}\circ \chi^G_{\rm Hdg}>> K_0(\co^G(X)) \\
@V c_*(g) VV  @V {T_y}_*(g) VV @VV td_*(g)V\\
H^{BM}_{ev}(X^g;\C) @< y=-1 << H^{BM}_{ev}(X^g;\C)[y] @> y=0 >> H^{BM}_{ev}(X^g;\C) ,
\end{CD}\end{equation}
where $rat$ associates to a mixed Hodge module complex its underlying constructible sheaf complex.
Compared to the non-equivariant version of \cite{BSY}, what is missing up to now is a suitable equivariant $L$-class theory $L_*(g)$ (corresponding to $y=1$) for singular spaces. However, these localized $L$-classes are available in the smooth context (by the classical $G$-signature theorem, see \cite{AS}), as well as for suitable global orbifolds (by work of Hirzebruch-Zagier, \cite{HZ,Za}), but see also \cite{CSW}. \\

An important application of the Hodge-theoretic Atiyah-Singer classes is the computation of the  homology Hirzebruch class ${{T_y}}_*(X/G)$ of the global quotient $X/G$. More precisely, in Section \S \ref{comporb} we prove the following result (even with suitable twisted coefficients):
\begin{theorem}\label{thintro} Let $G$ be a finite group acting by algebraic automorphisms on the complex quasi-projective variety $X$.  Let  $\pi^g:X^g \to X/G$ be the composition of the projection map $\pi:X \to X/G$ with the inclusion $i^g:X^g \hookrightarrow X$. Then
\begin{equation}\label{Za}
{{T_y}}_*(X/G)=\frac{1}{\vert G \vert} \sum_{g \in G} \pi^g_*{{T_y}}_*(X;g)
\end{equation}
and, for $X$ pure-dimensional,
\begin{equation}\label{Zag}
{{IT_y}}_*(X/G)=\frac{1}{\vert G \vert} \sum_{g \in G} \pi^g_*{{IT_y}}_*(X;g).
\end{equation}
\end{theorem}

As an application of (\ref{comm}),  one deduces that equation (\ref{Za}) can be viewed as an equality of homology classes in $H_{ev}^{BM}(X/G)\otimes\C[y]$. In particular, we are allowed to specialize to $y=-1, 0, 1$. For $y=-1$, formula (\ref{Za}) specializes by the left square of (\ref{comm}) to the identity:
\begin{equation}
 c_*(X/G)=\frac{1}{\vert G \vert} \sum_{g \in G} \pi^g_*c_*(X^g),
\end{equation}
where $c_*$ is the rationalized MacPherson homology Chern class from \cite{MP}. This  formula already appears in Ohmoto's work \cite{O}[Sect.2.2]. For $y=0$, formula (\ref{Za}) specializes to the Todd class formula
\begin{equation}\label{to}
{{td}}_*(X/G)=\frac{1}{\vert G \vert} \sum_{g \in G} \pi^g_*{{td}}_*(X;g),
\end{equation}
appearing in the works \cite{BFQ,M} if $X$ (and therefore also $X/G$, by \cite{KS}[Cor.5.4]) has at most {\em Du Bois singularities} (e.g., $X$ is smooth or has rational singularities). Here we use the identification 
$$td_*(X/G):=td_*([\OO_{X/G}])={T_0}_*(X/G), \quad \text{as well as} \quad td_*(X;g):=td_*(g)([\OO_X])={T_0}_*(X;g)$$ for $X$ with at most Du Bois singularities, see \cite{BSY} as well as Lem.\ref{db}. 
Finally, for $y=1$ and $X$ smooth, formula (\ref{Za}) yields the first two equalities of the identity
\begin{equation}\label{LZa}
{{T_1}}_*(X/G)=\frac{1}{\vert G \vert} \sum_{g \in G} \pi^g_*{{T_1}}_*(X;g)=\frac{1}{\vert G \vert} \sum_{g \in G} \pi^g_*{L}_*(X;g)=L_*(X/G),
\end{equation}
where for the last equality we assume $X$ is projective. Here $L_*(X/G)$ is the Thom-Hirzebruch $L$-class of the compact oriented rational homology manifold $X/G$, and the last equality is due to Zagier  \cite{Za}. In particular, (\ref{LZa}) yields the following result, supporting a conjecture from \cite{BSY}:
\begin{cor}\label{id} Let $X$ be a projective $G$-manifold, with $G$ a finite group of algebraic automorphisms. Then 
 \begin{equation}
{{T_1}}_*(X/G)=L_*(X/G).
\end{equation}
\end{cor}

If $X$ is {smooth and projective}, the result of Theorem \ref{thintro} is proved by Moonen  \cite{M}[p.170] for a parametrized Todd class $\tau_y(X/G)$, which he only could define for global projective orbifolds. So, by comparison, one gets the identification of his parametrized Todd class with the un-normalized Hirzebruch class of the quotient $X/G$:
\begin{equation}
\widetilde{{T_y}}_*(X/G)=\tau_y(X/G).
\end{equation}
This un-normalized Hirzebruch class is just a suitable ``twisting'' of the class ${T_y}_*(X/G)$ from the statement of Thm.\ref{thintro}, given in degree $2i$ by
$$(1+y)^{i}\cdot T_{y,i}(X/G) =: \widetilde{T}_{y,i}(X/G) \in H^{BM}_{2i}(X/G)\otimes \Q[y].$$

Theorem \ref{thintro} is just a very special case of a more general functorial result (see Sect.\ref{EH}, (\ref{refor})) about the {\em equivariant Hirzebruch class transformation}
\begin{equation}
 {\mt_y^G}_*:= \oplus_{g \in G} \ \frac{1}{\vert G \vert} \cdot {\mt_y}_*(g)\;:  K_0(\mh^G(X)) \to  H^{BM}_{ev,G}(X;\Q) \otimes \C[y^{\pm 1},(1+y)^{-1}]\:,
\end{equation}
with 
$$H^{BM}_{ev,G}(X;\Q):=\left( \oplus_{g \in G} \ H^{BM}_{ev}(X^g;\Q) \right)^G$$ 
the ``delocalized'' $G$-equivariant Borel-Moore homology of $X$.

\begin{theorem}\label{thintro2} Let $G$ be a finite group acting by algebraic automorphisms on the complex quasi-projective variety $X$.  Let  $\pi^g:X^g \to X/G$ be the composition of the projection map $\pi:X \to X/G$ with the inclusion $i^g:X^g \hookrightarrow X$. Then the following diagram commutes:
\begin{equation}\begin{CD}
K_0(\mh^G(X)) @> {\mt_y^G}_* >> H^{BM}_{ev,G}(X;\Q) \otimes \C[y^{\pm 1},(1+y)^{-1}]\\
@V [-]^G\circ \pi_* VV @VV \oplus_{g \in G} \;\pi^g_*  V \\
K_0(\mh(X/G)) @> {\mt_y}_* >> H^{BM}_{ev}(X/G;\Q) \otimes \C[y^{\pm 1},(1+y)^{-1}].
\end{CD}\end{equation}
Here $[-]^G: K_0(\mh^G(X'))\to K_0(\mh(X'))$ is induced by the exact projector $(-)^G: \mh^G(X')\to \mh(X')$ on the $G$-invariant subobject, for $X'=X/G$ 
a quasi-projective variety with a trivial $G$-action.
\end{theorem}

\begin{remark}\rm The use of the delocalized Borel-Moore homology in the corresponding result for the equivariant Todd class transformation
\begin{equation}
td_*^G:=\oplus_{g \in G} \ \frac{1}{\vert G \vert} \cdot td_*(g):  K_0({\rm Coh}^G(X)) \to  H^{BM}_{ev,G}(X;\Q) 
\end{equation}
appears already in the work of Baum-Connes \cite{BC} in relation to noncommutative geometry. Similarly, it fits with the Todd class transformation of Toen \cite{To}, defined in the context of  Deligne-Mumford stacks, with $\left(\bigsqcup_{g \in G} X^g \right)/G$ underlying the inertia stack of the quotient stack $X//G$.
\end{remark}

Theorem \ref{thintro} is  used in the authors' paper \cite{CMSSY} for obtaining generating series formulae for the homology Hirzebruch classes ${{T_y}}_*(X^{(n)})$ of the symmetric products $X^{(n)}:=X^n/{\Sigma_n}$ of a (possibly singular) complex quasi-projective algebraic variety $X$ (see also \cite{MS10}).
In fact, by making use of Theorem \ref{thintro2}, such generating series results are formulated in \cite{CMSSY} for characteristic classes of symmetric products of pairs $(X,\MC)$, with $\MC \in D^b\mh(X)$ (compare also with \cite{MS09,MSS}, where the case of Hodge numbers and Hodge-Deligne polynomials is discussed).\\

As another illustration of the use of the new Atiyah-Singer classes, we compute them in the context of monodromy problems, e.g., for varieties which fiber equivariantly (in the topological sense) over a connected algebraic manifold. The results obtained in Section \S \ref{mon} are characteristic class versions of the formulae described in \cite{eCMS}, and equivariant generalizations of some of the Atyah-Meyer type results from \cite{CLMS,CLMS2,MS}. Of particular importance is the result of Theorem \ref{thm1}, which can  be regarded as an Atiyah-Singer type theorem ``with twisted coefficients":
\begin{theorem}\label{thintro3}
Let $X$ be a  complex quasi-projective manifold, and $\LL$ a ``good'' variation of mixed Hodge structures on $X$. Let $G$ be a finite group of algebraic automorphisms of $(X,\LL)$. Then for any $g \in G$ we have:
\begin{equation}\label{fourintro}
{T_y}_*(X,\LL;g)
=ch_{(1+y)}(\chi_y(\VV) \vert_{X^g})(g) \cap {T_y}_*(X;g) ,\end{equation} 
where $\chi_y(\VV)\in K^0_G(X)[y^{\pm 1}]$ is an (equivariant) 
$\chi_y$-characteristic of the associated complex algebraic vector bundle $\VV:=\LL \otimes_{\Q} \mathcal{O}_X$, with its induced Hodge filtration $\FC^{\centerdot}$, and $ch_{(1+y)}(-)(g)$ is a suitable equivariant twisted Chern character.
\end{theorem}

Here, to keep the exposition simple, we use the shorter notion of {\it ``good"
variation of mixed Hodge structures} for a  variation which is 
graded polarizable, admissible and with quasi-unipotent monodromy at infinity.\\

As an application of the above results, we obtain the following  Atiyah-Meyer type result for twisted Hirzebruch classes of global orbifolds:
\begin{theorem}\label{thintro4}
Let $G$ be a finite group of algebraic automorphisms of the (pure-dimensional) quasi-projective manifold $M$, with $\pi:M \to X:=M/G$ the projection map. Let $\LL$ be a local system on $X$, which generically underlies a good variation of mixed Hodge structures defined on a Zariski dense open smooth subset $U \subset X$, and let $\VV:=\LL \otimes_{\Q} \mathcal{O}_{X}$ be the associated complex algebraic vector bundle. Then:
\begin{equation}\label{fiveintro}
{IT_y}_*(X,\LL)
=ch_{(1+y)}(\chi_y(\VV) ) \cap {IT_y}_*(X) ,\end{equation} 
where $ch_{(1+y)}(\chi_y(\VV) ) \in H^{ev}(X;\Q[y^{\pm 1}])$ corresponds, by definition,  to $ch_{(1+y)}(\chi_y(\pi^*\VV) ) \in H^{ev}(M;\Q[y^
{\pm 1}])^G$ under the isomorphism $$\pi^*:H^{ev}(X;\Q[y^{\pm 1}]) \overset{\sim}{\to} H^{ev}(M;\Q[y^{\pm 1}])^G \subset H^{ev}(M;\Q[y^{\pm1}]).$$ Here $\pi^*(\LL)$ underlies a good variation of mixed Hodge structures on all of $M$, so that $\chi_y(\pi^*\VV)$ is the $\chi_y$-characteristic of the associated complex algebraic vector bundle $\pi^*\VV:=\pi^*\LL \otimes_{\Q} \mathcal{O}_M$, with its induced Hodge filtration $\FC^{\centerdot}$.
\end{theorem}

Note that we cannot directly define $\chi_y(\VV) \in K^0(X)[y^{\pm 1}]$ as in  \cite{CLMS,CLMS2,MS} for the smooth context, since the Hodge filtration of $\VV|_U$ does not necessarily extend as a filtration by sub-vector bundles of $\VV$ to all of $X$. But this applies to $\pi^*\VV$ on the {\em smooth} variety $M$. So our definition above only works in the global orbifold situation $X=M/G$, and it is not available for more singular spaces. If, moreover, in the context of the above theorem, we assume that $M$ projective and $\LL$ is generically a good variation of {\em pure} Hodge structures, equation (\ref{fiveintro}) reduces for  $y=1$ to the first equality of:
\begin{equation}\label{id2}
{IT_1}_*(X,\LL)
=ch_{(2)}([\LL]_K ) \cap {L}_*(X) = L_*(X,\LL),
\end{equation} 
while the second equality is the Atiyah-Meyer formula of \cite{BCS} for twisted $L$-classes. Here, $[\LL]_K$ is the $K$-theoretical signature class of Meyer, associated to a suitable duality structure of $\LL$ coming from the polarization of $\LL|_U$. Indeed, $\pi^*[\LL]_K=[\pi^*\LL]_K$ can be identified with $\chi_1(\pi^*(\VV))$, by \cite{Sch3}[Cor.3.8].
Also, as $X$ is a rational homology manifold, we have that ${IT_1}_*(X)={T_1}_*(X)=L_*(X)$, with the last identification following from Cor.\ref{id}. In particular, equation (\ref{id2}) yields an equality between the twisted characteristic classes ${IT_1}_*(X,\LL)$ and $L_*(X,\LL)$.\\

If, more generally, we assume in the context of Thm.\ref{thintro4} that $\LL$ is only generically defined on $X$, but $\pi^*(\LL)$ extends to all of $M$, then formula (\ref{fiveintro}) does not necessarily hold. In the main body of the paper, we even provide a formula for the correction terms appearing in this situation, see formulae (\ref{def2}) and (\ref{def3}) of Remark \ref{defect}. It would be interesting to compare this result with the corresponding ``defect formula" of \cite{Ba} in the context of twisted $L$-classes.
\\

If, in the context of Thm.\ref{thintro4}, we additionally assume that $M$ is connected and $\pi^*(\LL)$ is constant, e.g., $M$ is simply-connected, equation (\ref{fiveintro}) reduces, by rigidity, to the multiplicative formula:
\begin{equation}\label{sixintro}
{IT_y}_*(X,\LL)
=\chi_y(\LL_x) \cdot {T_y}_*(X) ,\end{equation} where $\LL_x$ is the stalk of $\LL$ at some generic smooth point of $X$, with its induced mixed Hodge structure. Note that $\LL$ (and therefore also $\pi^*(\LL)$) is constant if $X$ is simply-connected. E.g., this is the case for the weighted projective spaces. So the calculation of the twisted Hirzebruch classes ${IT_y}_*(X,\LL)$ amounts in this case to understanding the Hirzebruch classes ${T_y}_*(X)$ of the global orbifold $X$.
\\

We conclude this introduction with an example, due to Moonen (\cite{M}[p.176]), on the calculation of (un-normalized) Hirzebruch classes 
$\widetilde{{T_y}}_*(\mathbb{P}^n(\textbf{w}))=\tau_y(\mathbb{P}^n(\textbf{w}))$
of weighted projective spaces $\mathbb{P}^n(\textbf{w})$. While weighted projective spaces of the same dimension have the same Hodge polynomials (e.g., see \cite{eCMS}[Rem.3.3(iv)]), Thm.\ref{thintro} can be used to show that these spaces are in fact distinguished by their Hirzebruch classes.

\begin{example}[Hirzebruch classes of weighted projective spaces]
Let $\mathbb{P}^n(\textbf{w})=\mathbb{P}^n/{G(\textbf{w})}$ be the weighted projective space, with $G(\textbf{w})=G({w}_0) \times \cdots \times G({w}_n)$ and $G({m})$ the multiplicative group of $m$-th roots of unity, acting by multiplication on the corresponding homogeneous coordinates.
Then 
\begin{equation}
(1+y)\cdot \pi^*\tilde{T}^*_y(\mathbb{P}^n(\textbf{w}))=\frac{deg(\pi)}{\vert G \vert} \sum_{0 \leq \alpha < 2\pi} \prod_{j=0}^n w_j x \frac{1+ye^{-w_j(x+i\alpha)}}{1-e^{-w_j(x+i\alpha)}}
\in H^{ev}( \mathbb{P}^n)\otimes \C[y]
\end{equation}
for $\pi: \mathbb{P}^n \to \mathbb{P}^n(\textbf{w})$ the projection map, and $x \in H^2(\mathbb{P}^n)$ the cohomology generator in degree $2$ dual to the fundamental class
of $\mathbb{P}^{n-1}\subset  \mathbb{P}^n$, with $x^{n+1}=0$. Here $d:=\frac{|G|}{deg(\pi)}$ is the greatest common divisor of the $w_j$ ($j=0,\dots,n$). Note that
$$\pi^*: H^{BM}_{ev}(\mathbb{P}^n(\textbf{w}))\otimes \Q[y]\simeq H^{ev}(\mathbb{P}^n(\textbf{w}))\otimes\Q[y]\to H^{ev}(\mathbb{P}^n)\otimes\Q[y]
\subset H^{ev}(\mathbb{P}^n)\otimes\C[y]$$
is injective, with image the $G(\textbf{w})$-invariant cohomology classes in $H^{ev}(\mathbb{P}^n)\otimes\Q[y]$.
Finally, only finitely many $\alpha$ contribute to the sum above, since 
$$\prod_{j=0}^n w_j x \frac{1+ye^{-w_j(x+i\alpha)}}{1-e^{-w_j(x+i\alpha)}}
\in \left(x^{n+1}\cdot \Z[[x]]\right)\otimes \C[y]\:,$$
if $w_j\cdot \alpha$ is not a multiple of $2\pi$ for all $0\leq j \leq n$. 
\end{example}

\section{Preliminaries}

\subsection{Characteristic classes in cohomology}\label{not}

Let $X$ be a complex (algebraic) manifold, and  $\Xi$ a rank $q$ complex (algebraic) vector bundle on $X$.
In what follows we say that a (total) cohomology characteristic class $\Phi$ of $\Xi$ is defined by a power series $f(\alpha)\in R[[\alpha]]$, with $R$ a $\Q$-algebra, if we have the following relation: $\Phi(\Xi)=\prod_{i=1}^q f(\alpha_i)\in H^{ev}(X)\otimes R$, where the  $\{ \alpha_i\}_{i=1}^q$ are the Chern roots of $\Xi$. 

In order to set the notations for the rest of the paper, let us introduce the following characteristic classes of a complex (algebraic) vector bundle $\Xi$ on the complex manifold $X$ (compare \cite{HZ}[pp.40-42]):
\begin{enumerate}
\item[(i)] The Chern class $c(\Xi)$, given by the power series $f(\alpha)=1+{\alpha}\in \Q[[\alpha]]$.
\item[(ii)] The $L$-class $L(\Xi)$, given by $f(\alpha)=\frac{\alpha}{\tanh \alpha}\in \Q[[\alpha]]$.
\item[(iii)] The Todd class $td(\Xi)$, given by $f(\alpha)=\frac{\alpha}{1-e^{-\alpha}}\in \Q[[\alpha]]$.
\item[(iv)] The ``normalized" Hirzebruch class $T_y(\Xi)$, given by $f_y(\alpha)=\frac{\alpha(1+y)}{1-e^{-\alpha(1+y)}}-\alpha y \in \Q[y][[\alpha]]$. Note that for various values of the parameter $y$ we obtain $T_0(\Xi)=td(\Xi)$, $T_1(\Xi)=L(\Xi)$ and $T_{-1}(\Xi)=c(\Xi)$.
\item[(v)] The ``un-normalized" class $\widetilde{T}_y(\Xi)$, given by $\tilde{f}_y(\alpha)=\frac{\alpha
(1+ye^{-\alpha})}{1-e^{- \alpha}}\in \Q[y][[\alpha]]$\footnote{The attribute ``(un-)normalized" refers to the fact that the defining power series $f_y(\alpha)$ and resp. $\tilde{f}_y(\alpha)$  satisfy:  ${f}_y(0)=1$, while $\tilde{f}_y(0)=1+y$.}. We also have that $\widetilde{T}_0^*(\Xi)=td(\Xi)$.
\item[(vi)] The class ${\U}_{\theta}(\Xi)$, given by $f(\alpha)=(1-e^{-\alpha-i\theta})^{-1}\in \C[[\alpha]]$, where $\theta \in \mathbb{R}$ is not divisible by $2\pi$.
\item[(vii)] The class $T_y^{\theta}(\Xi)$, given by $f(\alpha)=\frac{1+ye^{-i\theta-\alpha(1+y)}}{1-e^{-i\theta-\alpha(1+y)}}\in \C[y][[\alpha]]$, with $y$ and $\theta$ as before. Thus $T_0^{\theta}(\Xi)={\U}_{\theta}(\Xi)$.
\item[(viii)] The class $\widetilde{T}_y^{\theta}(\Xi)$, given by $f(\alpha)=\frac{1+ye^{-i\theta-\alpha}}{1-e^{-i\theta-\alpha}}\in \C[y][[\alpha]]$. So $\widetilde{T}_0^{\theta}(\Xi)={\U}_{\theta}(\Xi)$.
\end{enumerate}

As a convention, if $\Phi$ is one of the above characteristic classes, we write $\Phi(X)$ for the class of the holomorphic tangent bundle $T_X$ of $X$.
For a holomorphic vector bundle $\Xi$ on the complex manifold $X$, we let $\Omega(\Xi)$ denote the sheaf of germs of holomorphic sections of $\Xi$. In what follows, we omit the symbol $\Omega(-)$, and simply write $H^i(X;\Xi)$ in place of $H^i(X;\Omega(\Xi))$. 

\subsection{Atiyah-Singer classes of complex manifolds}\label{ASm}

Let $g$ be an automorphism of the pair $(X,\Xi)$, where $X$ is a compact complex manifold and $\Xi$ is a holomorphic bundle on $X$. Then $g$ induces automorphisms on the global sections $\Gamma(X;\Xi)$ of $\Xi$, and also on the higher cohomology groups $H^i(X;\Xi)$. The {\it $g$-holomorphic Euler characteristic} of $\Xi$ over $X$ is defined by:
\begin{equation}\label{1}\chi(X,\Xi;g):=\sum_i (-1)^i \cdot {\rm trace} \left(g \vert H^i(X;\Xi)\right).\end{equation}
The automorphism $g:X \to X$ also induces an automorphism of the holomorphic cotangent bundle $T^*_X$, so an automorphism of $(X,\Xi)$ induces an automorphism of the pair $(X, \Xi \otimes \Lambda^p T^*_X)$, $p \in \Z$. The following polynomial invariant is a parametrized version of $\chi(X,\Xi;g)$:
\begin{equation}\label{2}\chi_y(X,\Xi;g):=\sum_{p \geq 0} \chi(X,\Xi \otimes \Lambda^pT^*_X;g)
\cdot y^p.\end{equation}
Assume moreover that a finite group $G$ acts holomorphically on the  complex manifold $X$. Then for $g \in G$, the fixed-point set $X^g:=\{x \in X \vert \ gx=x\}$ is a complex submanifold of $X$ and $g$ acts on the normal bundle $N^g$ of $X^g$ in $X$. Since $X$ is complex, we have a decomposition $$N^g=\bigoplus_{0<\theta<2\pi} N^g_{\theta},$$ where each sub-bundle $N^g_{\theta}$ inherits a complex structure from that of  $X$, and $g$ acts as $e^{i\theta}$ on $N^g_{\theta}$.  

Recall that if $\Xi \in K_G(X)$ is a $G$-equivariant complex vector bundle on the topological space $X$ on which a finite group $G$ acts trivially, then we can write $\Xi$ as a sum $\Xi=\sum_i \Xi_i \otimes \chi_i$, for $\Xi_i \in K(X)$ and $\chi_i \in R(G)$, where $K(X)$ denotes the Grothendieck group of $\C$-vector bundles on $X$ and $R(G)$ is the complex representation ring of $G$ (see \cite{Se}[Prop.2.2]). We then define \begin{equation}\label{equivchern}ch(\Xi)(g):=\sum_i ch(\Xi_i) \cdot \chi_i(g) \in H^{ev}(X;\C),\end{equation} 
with $ch: K(X)\to H^{ev}(X;\Q)$ the Chern character and $\chi_i(g)\in \C$ the corresponding trace of $g$ (compare also with \cite{BFQ} for a corresponding definition in the algebraic context). So this transformation only depends on the cyclic subgroup of $G$ generated by $g$. In what follows, we apply this fact to the space $X^g$ on which $g$ acts trivially.

We can now state the following important result (compare also with \cite{HZ}[p.51/52]):
\begin{theorem}\label{AS}(The Atiyah-Singer holomorphic Lefschetz theorem, \cite{AS})\newline
Let $\Xi$ be a holomorphic vector bundle on a compact complex manifold $X$ and $g$ a finite order automorphism of $(X,\Xi)$. Then
\begin{equation}\label{AS0}
\chi(X,\Xi;g)=\langle ch(\Xi \vert_{X^g})(g) \cdot td(X^g) \cdot \prod_{0<\theta <2\pi} \U_{\theta}(N^g_{\theta}),[X^g] \rangle.
\end{equation}
Or, in parametrized version,
\begin{eqnarray}\label{AS1}
\chi_y(X,\Xi;g) &=& \langle ch(\Xi \vert_{X^g})(g) \cdot \widetilde{T}_y(X^g) \cdot \prod_{0<\theta <2\pi} \widetilde{T}_y^{\theta}(N^g_{\theta}),[X^g] \rangle \\
&=&\langle ch_{(1+y)}(\Xi \vert_{X^g})(g) \cdot {T}_y(X^g) \cdot \prod_{0<\theta <2\pi} {T}_y^{\theta}(N^g_{\theta}),[X^g] \rangle,
\end{eqnarray}
(the dot stands for the cup product in cohomology, while $\langle-,-\rangle$ denotes the non-degenerate bilinear evaluation pairing)
where, for a complex bundle $\Xi$ with Chern roots $\{\alpha_i\}_i$, we set $ch_{(1+y)}(\Xi):=\sum_{i=1}^{{\rm rk}(\Xi)} e^{(1+y)\alpha_i} \in H^{ev}(X)\otimes\Q[y]$.
\end{theorem}

\begin{defn} The (total) Atiyah-Singer characteristic class of the pair $(X,g)$, for  $g$  a finite order automorphism of a complex manifold $X$, is defined as 
\begin{equation} T^*_y(X;g):={T}_y(X^g) \cdot \prod_{0<\theta <2\pi} {T}_y^{\theta}(N^g_{\theta}) \in H^{ev}(X^g) \otimes \C[y],
\end{equation}
or, in its ``un-normalized" form, 
\begin{equation} \widetilde{T}^*_y(X;g):=\widetilde{T}_y(X^g) \cdot \prod_{0<\theta <2\pi} \widetilde{T}_y^{\theta}(N^g_{\theta}) \in H^{ev}(X^g) \otimes \C[y],
\end{equation}
Similarly, we let 
\begin{equation}\label{td} td^*(X;g):= td(X^g) \cdot \prod_{0<\theta <2\pi} \U_{\theta}(N^g_{\theta})  \in H^{ev}(X^g)\otimes \C,\end{equation}
and note that $T^*_0(X;g)=\widetilde{T}^*_0(X;g)=td^*(X;g)$.
\end{defn}

\begin{remark}\label{23}\rm
It follows from the above definition that the classes $\widetilde{T}^*_y(X;g)$ and ${T}^*_y(X;g)$ are two parametrized versions of $td^*(X;g)$, which differ just by suitable powers of $1+y$ in each degree. Other important special values of these parametrized classes include, at $y=1$:
\begin{equation}\label{L-class}
L(X;g):={T}^*_1(X;g) \quad \text{and} \quad \tilde{L}(X;g):=\widetilde{T}^*_1(X;g)
\end{equation}
appearing in the equivariant signature theorem \cite{AS,HZ,Za}.
Also, for $y=-1$, we get the total (resp. top) Chern class of the fixed-point set (in $H^{ev}(X^g;\Q)$):
\begin{equation}\label{chern}
c(X^g)={T}^*_{-1}(X;g) \quad \text{resp.} \quad c^{\rm top} (X^g)=\widetilde{T}^*_{-1}(X;g).
\end{equation}
Note that there is no essential difference between the classes $\widetilde{T}^*_y(X;g)$ and ${T}^*_y(X;g)$, except for the specialization at $y=-1$. 
\end{remark}

It follows from the Atiyah-Singer holomorphic Lefschetz theorem that if $X$ is a compact complex manifold then the equivariant $\chi_y$-genus of $X$ defined by $(\ref{2})$ is the degree of the top-dimensional component of the (un-normalized) Atiyah-Singer class, that is,
\begin{equation}\label{26}
\chi_y(X;g)=\langle T^*_y(X;g), [X^g]\rangle=\langle \widetilde{T}^*_y(X;g), [X^g]\rangle.
\end{equation} 
So, by the above identifications and the equivariant signature theorem (cf. \cite{AS,HZ,Za}), this yields for $y=1$ that:
\begin{equation}
\chi_1(X;g)=\langle L(X;g), [X^g]\rangle=\langle \widetilde{L}(X;g), [X^g]\rangle =\sigma(X;g),
\end{equation}
with $\sigma(X;g)$ the $g$-equivariant signature. Also, for $y=-1$, (\ref{26}) just gives the Lefschetz fixed-point formula:
\begin{equation}
\chi_{-1}(X;g)=\langle c(X^g), [X^g]\rangle=\langle c^{\rm top}(X^g), [X^g]\rangle =\chi(X^g),
\end{equation}
with $\chi(X^g)$ the topological Euler characteristic of $X^g$. Note that 
$$\chi_{-1}(X;g)=\sum_{p \geq 0} \chi(X,\Lambda^pT^*_X;g)
\cdot (-1)^p=\sum_i (-1)^i \cdot {\rm trace} \left(g \vert H^i(X;\C)\right)$$ 
calculates the ``topological'' trace of $g$ on the complex cohomology of $X$. \\


One of the aims of this note is to define (motivic and Hodge-theoretic) Atiyah-Singer classes $$\widetilde{{T_y}}_*(X;g), \ \ {T_y}_*(X;g) \in H_{ev}^{BM}(X^g;\C)[y]$$ for any (possibly {\it singular}) quasi-projective variety $X$ acted upon by a finite group $G$ of algebraic automorphisms, so that these classes  satisfy the  {\it normalization} property  asserting that if $X$ is non-singular then: \begin{equation} \widetilde{{T_y}}_*(X;g)=\widetilde{T}^*_y(X;g) \cap [X^g] \ \ \ \text{and} \ \ \  {T_y}_*(X;g)= T^*_y(X;g) \cap [X^g]. \end{equation}
In the case when $X$ is a projective (but possibly singular) variety, by pushing down to a point we shall recover the equivariant $\chi_y$-genus studied in \cite{eCMS}. In other words, the polynomial   \begin{equation} \chi_y(X;g):=\sum_{i,p}(-1)^i {\rm trace} \left(g \vert {\rm Gr}^p_F H^i(X;\C) \right) \cdot (-y)^p\end{equation} (for $F^{\centerdot}$  the Hodge filtration of the canonical Deligne mixed Hodge structure on $H^*(X;\Q)$) should coincide with the degree of the zero-dimensional component of the Atiyah-Singer class:
\begin{equation} \chi_y(X;g)=\int_{[X^g]} \widetilde{{T_y}}_*(X;g)=\int_{[X^g]} {T_y}_*(X;g). \end{equation}

\subsection{Motivic Chern and Hirzebruch classes}\label{nonequiv}  In order to better motivate our construction of Atiyah-Singer classes in the singular context, we provide here a quick review of the main properties of homology Hirzebruch classes of (possibly singular) complex algebraic varieties, as developed by Brasselet, Sch\"urmann and Yokura in \cite{BSY} (see also \cite{CMS,MS,Sch3}).

Let $X$ be a complex algebraic variety. By building on Saito's functors (cf. \cite{Sa})
\begin{equation} gr^F_pDR: D^b\mh(X) \to D^b_{coh}(X)\end{equation}
(for  $D^b_{coh}(X)$ the bounded
derived category of sheaves of $\mathcal{O}_X$-modules with coherent
cohomology sheaves), one first defines a {\em motivic Chern class transformation}
\begin{equation}
 \mc: K_0(\mh(X)) \to K_0(D^b_{\rm coh}(X)) \otimes \Z[y^{\pm 1}]=K_0(\co(X)) \otimes \Z[y^{\pm 1}].
\end{equation}
After composing this with (a modified version of) the Baum-Fulton-MacPherson Todd class transformation \cite{BFM}  \begin{equation}td_*:K_0(\co(X)) \to H^{BM}_{ev}(X;\Q),\end{equation} linearly extended over $\Z[y^{\pm 1}]$,
the authors of \cite{BSY} (see also the refinements in \cite{Sch3}) defined the so-called {\em Hirzebruch class transformation} 
\begin{equation}{\mt_y}_*:K_0(\mh(X)) \to H^{BM}_{ev}(X)\otimes \Q[y^{\pm 1}] \subset H^{BM}_{ev}(X)\otimes \Q[y^{\pm 1},(1+y)^{-1}],\end{equation} which assigns classes in the Borel-Moore homology of $X$ (with polynomial coefficients) to any ($K_0$-class of a) mixed Hodge module on $X$. By its construction, the transformation ${\mt_y}_*$ commutes with push-down for proper maps.\\

For example, by applying these transformations to the class of the constant Hodge sheaf $\Q^H_X$, one defines the motivic Chern and resp. Hirzebruch class $mC_y(X)$, resp., ${T_y}_*(X)$ of $X$. If $X$ is smooth, this class is Poincar\'e dual to the total $\lambda$-class of the cotangent bundle, resp., the cohomology Hirzebruch class $T_y(X)$ defined in \S\ref{not}.

Over a point space, both transformations $\mc$ and ${\mt_y}_*$  coincide with the
$\chi_y$-genus ring homomorphism $\chi_y:K_0(\ms^p) \to
\Z[y,y^{-1}]$, which is defined on the Grothendieck group of (graded) polarizable mixed Hodge structures by 
\begin{equation} \chi_y([H]):=\sum_p {\rm dim} Gr^p_F(H \otimes \C) \cdot (-y)^p,
\end{equation}
for $F^{\centerdot}$ the Hodge filtration of $H \in \ms^p$. So if $X$ is a compact variety, by pushing down to a point it follows immediately that the degree of the zero-dimensional component of the homology Hirzebruch class ${T_y}_*(X)$ is the Hodge polynomial of $X$, defined as $\chi_y(X):=\sum_i(-1)^i\cdot \chi_y([H^i(X;\Q)]).$\\

The corresponding {\em motivic Hirzebruch class transformation} on the relative Grothendieck group of algebraic varieties over $X$ is defined in \cite{BSY,Sch3} as: 
$${T_y}_*:=\chi_{\rm Hdg} \circ {\mt_y}_*:K_0(var/X) \to H^{BM}_{ev}(X)\otimes \Q[y],$$ with $\chi_{\rm Hdg}: K_0(var/X) \to K_0(\mh(X))$ given by $[f:Y \to X] \mapsto [f_!\Q^H_Y]$. Then also ${T_y}_*$ commutes with push-down for proper maps, and it unifies as such a transformation in a functorial sense the well-known Chern classes of MacPherson \cite{MP}, Todd classes of Baum, Fulton and MacPherson \cite{BFM}, and $L$-classes of Goresky-MacPherson \cite{GM1}, Cheeger \cite{Che} and Cappell-Shaneson \cite{CS1}, respectively. But on the space level, i.e., for the homology Hirzebruch class ${T_y}_*(X):={T_y}_*([id_X])$ of $X$, more care is needed for these identifications.  For $y=-1$, the motivic Hirzebruch class of $X$ specializes into the rationalized MacPherson-Chern class $c_*(X) \otimes \Q$. For $y=0$, but for $X$ with at most Du Bois singularities, one recovers the Todd class $td_*(X)$ of Baum-Fulton-MacPherson. For $y=1$ and $X$ a compact rational homology manifold, it is only conjectured that one obtains the Thom-Milnor $L$-class $L_*(X)$.

One of the main purposes of this paper is to develop analogous equivariant theories of characteristic classes, and to use these new theories in order to understand characteristic classes of global quotient varieties (e.g., symmetric products of varieties).

\subsection{Background on the Lefschetz-Riemann-Roch transformation} An essential ingredient in our definition of Atiyah-Singer classes is the Lefschetz-Riemann-Roch transformation of Baum-Fulton-Quart \cite{BFQ} and Moonen \cite{M}. We recall here some of the main properties of this transformation. 

Let $X$ be a quasi-projective $G$-variety, for $G$ a finite group of algebraic automorphisms of $X$. Denote by $K_0(\co^G(X))$ the Grothendieck group of the abelian category $\co^G(X)$ of $G$-equivariant coherent algebraic sheaves on $X$. For each $g \in G$, the Lefschetz-Riemann-Roch transformation 
\begin{equation}
td_*(g)(-):K_0(\co^G(X)) \to H^{BM}_{ev}(X^g;\C)
\end{equation}
takes values in the even-degree part of the Borel-Moore homology of the fixed-point set $X^g$, and satisfies the following properties:
\begin{itemize}
\item {\it covariance:} $td_*(g)(-)$ is a natural transformation, in the sense that for a proper $G$-morphism $f:X \to Y$ of quasi-projective varieties the following diagram commutes:
\begin{equation}
\begin{CD}
 K_0(\co^G(X)) @> td_*(g) >> H^{BM}_{ev}(X^g;\C)\\
@V f_! VV @VV f^g_* V \\
K_0(\co^G(Y)) @> td_*(g) >>
H^{BM}_{ev}(Y^g;\C) 
\end{CD}
\end{equation}
Here $f_![\FC]:=\sum_{i\geq 0} (-1)^i \cdot [R^if_*\FC]$, and $f^g:X^g \to Y^g$ is induced by the $G$-map $f$.
\item {\it module:} For every $G$-space $X$, there is a commutative diagram:
\begin{equation}
 \begin{CD}
 K^G_0(X) \times  K_0(\co^G(X)) @> ch(-|_{X^g})(g) \times td_*(g) >> 
H^{ev}(X^g;\C) \times H^{BM}_{ev}(X^g;\C)\\
@V \otimes VV @VV \cap V\\
K_0(\co^G(X)) @> td_*(g) >> H^{BM}_{ev}(X^g;\C),
 \end{CD}
\end{equation}
with $K^G_0(X)$ the Grothendieck group of algebraic $G$-vector bundles. In particular,
if $\Xi$ is an algebraic  $G$-vector bundle on $X$,  then 
\begin{equation}\label{l1}td_*(g)([\OO(\Xi)])=ch(\Xi \vert_{X^g})(g) \cap td_*(X;g).\end{equation}
\item {\it exterior product:} Let $X$ and $X'$ be algebraic $G$- and $G'$-spaces, respectively. Then for $g\in G$ and $g'\in G'$, one has a commutative diagram:
\begin{equation}\label{multL}
 \begin{CD}
 K_0(\co^G(X)) \times K_0(\co^{G'}(X')) @> td_*(g)\times td_*(g') >>
H^{BM}_{ev}(X^g;\C)\times H^{BM}_{ev}(X'^{g'};\C)\\
@V \boxtimes VV @VV \times  V\\
K_0(\co^{G\times G'}(X\times X')) @> td_*\left( (g,g') \right)>> H^{BM}_{ev}(X^g\times X'^{g'};\C).
 \end{CD}
\end{equation}
\item  {\it normalization:} Assume $X$ is smooth. Then the natural map $K^0_G(X) \to K_0(\co^G(X))$ is an isomorphism, and 
\begin{equation}\label{l0}  td_*(g)(-)=ch(-|_{X^g})(g) \cdot td^*(X;g) \cap [X^g],
\end{equation}
for $td^*(X;g)$ defined by equation (\ref{td}). In particular,
\begin{equation} td_*(X;g):=td_*(g)([\OO_X])=td^*(X;g) \cap [X^g].\end{equation}
In general, for a possibly singular quasi-projective variety $X$ we set 
\begin{equation} td_*(X;g):=td_*(g)([\OO_X])\in H^{BM}_{ev}(X^g;\C).\end{equation}
\item {\it degree:} Assume $X$ is projective, so the constant map to a point is proper.  Pushing $\FC \in \co^G(X)$ down to a point gives by the covariance property that
\begin{equation}\label{l2} \chi(X,\FC;g)=\int_{[X]} td_*(g)([\FC]).\end{equation}
In particular, if $\FC$ is locally free, then (\ref{l1}) yields:
\begin{equation} \chi(X,\FC;g)=\langle ch(\FC \vert_{X^g})(g), td_*(X;g) \rangle.\end{equation}
\end{itemize}
\begin{remark}\rm
 In addition, the transformation $td_*(g)(-)$ commutes with restriction to open subsets, and for $g=id$ the identity element, it reduces to the complexified non-equivariant Todd transformation $td_*$ of Baum-Fulton-MacPherson \cite{BFM}. Note that \cite{BFQ} also constructs a  $K$-theoretic resp. Chow-group version of these transformations,
even for a more general notion of equivariant sheaves, but under the assumption that the fixed-point set $X^g$ is projective. This last assumption is not needed for the homology version used here, as proved in \cite{M}. Finally, the exterior product property is stated here in slightly more general terms than in \cite{BFQ,M}, but their proofs apply without modifications to the more general context mentioned above, because the transformation $td_*(g)(-)$ only depends on the action of the cyclic subgroup generated by $g$.
\end{remark}

\section{Equivariant motivic Chern classes}\label{equivc}
\subsection{Construction} We first construct a characteristic class transformation $\mc^G$ for the algebraic action of a finite group $G$ on a quasi-projective complex algebraic variety $X$:
\begin{multline}\label{de} \mc^G:K_0(\mh^G(X)))= K_0(D^{b,G} \mh(X)) \to \\ \to K_0(D_{\rm coh}^{b,G}(X))\otimes \Z[y^{\pm 1}]=
K_0({\rm Coh}^G(X)) \otimes \Z[y^{\pm 1}].
\end{multline}
Here we use the following notations:
\begin{itemize}
 \item $D^{b,G} \mh(X)$ is the category of $G$-equivariant objects in the derived category of algebraic mixed Hodge modules;
\item $\mh^G(X)$ is the abelian category of $G$-equivariant algebraic mixed Hodge modules on $X$;
\item $D_{\rm coh}^{b,G}(X)$ is the category of $G$-equivariant objects in the derived category $D_{\rm coh}^{b}(X)$ of bounded complexes of $\OO_X$-sheaves with coherent cohomology;
\item ${\rm Coh}^G(X)$ is the abelian category of $G$-equivariant coherent $\OO_X$-sheaves
\end{itemize}

\begin{remark}\rm
In all cases above, a $G$-equivariant element $\MC$ is just an element in the underlying additive category (e.g., $D^b\mh(X)$), with a $G$-action given by isomorphisms
$$\psi_g: \MC \to g_* \MC \quad (g\in G),$$  
such that $\psi_{id}=id$ and $\psi_{gh}=g_*(\psi_h)\circ \psi_g$ for all $g,h \in G$ (see 
\cite{MS09}[Appendix]). Note that many references (e.g., \cite{BFQ} or \cite{eCMS}) work with the corresponding isomorphisms $g^* \MC\to \MC$ defined by adjunction,
which are more natural for contravariant theories such as $K^0(X)$ or variations of mixed Hodge structures.
Also
these ``weak equivariant derived categories'' $D^{b,G}(-)$, are simpler and different than the corresponding equivariant derived categories in the sense of \cite{BL}, e.g., they are not triangulated in general. Nevertheless, one can define a suitable Grothendieck group, using ``equivariant distinguished triangles'' in the underlying derived category $D^b(-)$, and get  isomorphisms (cf. Lem.\ref{last}) $$ K_0(D^{b,G} \mh(X))=K_0(\mh^G(X))) \quad \text{and} \quad K_0(D_{\rm coh}^{b,G}(X))=
K_0({\rm Coh}^G(X)),$$ as explained in detail in Appendix A. This is enough for the purpose of this paper, since our characteristic class transformations are defined on the level of Grothendieck groups. Furthermore, as shown in Appendix A, this approach easily allows one to lift the usual calculus of Grothendieck functors like (proper) push-forward, exterior product and (smooth) pullback from the underlying non-equivariant to the equivariant context (similarly to the calculus of Grothendieck functors for the
equivariant derived categories in the sense of \cite{BL}). 
\end{remark}

Since Saito's natural transformations of triangulated categories (cf. \cite{Sa}) 
$${\rm gr}^F_p DR:D^b\mh(X) \to D^b_{\rm coh}(X)$$ commute with the push-forward $g_*$ induced by each $g \in G$ (since $g \in {\rm Aut}(X)$, so $g:X \to X$ is a proper map) we obtain an equivariant transformation (cf. Ex.\ref{66} in the Appendix) $${\rm gr}^F_p DR^G:D^{b,G}\mh(X) \to D^{b,G}_{\rm coh}(X).$$
Note that for a fixed $\MC \in D^{b,G} \mh(X)$, one has that ${\rm gr}^F_p DR^G(\MC)=0$ for all but finitely many $p \in \Z$. 
Therefore, we can now consider the cohomology $\left[H^*({\rm gr}^F_p DR^G (\MC))\right] \in K_0({\rm Coh}^G(X))$. This yields the following:
\begin{defn} The \underline{$G$-equivariant motivic Chern class transformation} 
$$\mc^G:K_0(\mh^G(X)) \to 
K_0({\rm Coh}^G(X)) \otimes \Z[y^{\pm 1}]$$
is defined by:
\begin{equation}\label{d}
\mc^G([\MC]):=\sum_{i,p}(-1)^i \left[H^i({\rm gr}^F_{-p} DR^G (\MC))\right] \cdot (-y)^p .
\end{equation}
 \end{defn}

\subsection{Properties}
By construction, the transformation $\mc^G$ commutes with proper push-down and restriction to open subsets. Moreover, for a subgroup $H$ of $G$, the transformation commutes with the obvious restriction functors ${\rm Res}^G_H$. For the trivial subgroup, this is just the forgetful functor ${\it For}:={\rm Res}^G_{\{id\}}$. If $G=\{id \}$ is the trivial group, $\mc^G$ is just the (non-equivariant) motivic Chern class transformation of \cite{BSY,Sch3}.\\

Our approach based on weak equivariant complexes of mixed Hodge modules allows us to formally extend most of the results (and their proofs) from \cite{BSY} and \cite{Sch3}[Sect.4,5] from the non-equivariant to the equivariant context considered here. For this type of results, we only give a brief account. For example:
\begin{lemma}\label{db} Let $G$ be a finite group of algebraic automorphisms of a complex quasi-projective variety $X$, with at most Du Bois singularities.  Then $${\rm MHC}^G_0([\Q^H_X])=[\OO_X] \in K_0({\rm Coh}^G(X)),$$ as given by the class of the structure sheaf with its canonical $G$-action.
\end{lemma}
\begin{proof}
By \cite{Sa1}, there is a canonical morphism $$\OO_X \to {\rm gr}^F_{0} DR (\Q^H_X)$$ in $D^b_{\rm coh}(X)$, which is an isomorphism for $X$ with at most Du Bois singularities. So if, in addition, $X$ has a action of a finite group $G$ as above, then this becomes a $G$-equivariant isomorphism. In particular, $$[\OO_X]=[{\rm gr}^F_{0} DR (\Q^H_X)] \in K_0({\rm Coh}^G(X)).$$
\end{proof}

As another instance, let $X$ be a complex algebraic manifold of pure dimension $n$, together with a ``good'' variation $\LL$ of mixed Hodge structures (i.e., graded polarizable, admissible and with quasi-unipotent monodromy at infinity). This corresponds as in \cite{Sch3}[Ex.4.2] to a (shifted) {\it smooth} mixed Hodge module $\LL^H$ with underlying rational sheaf complex $rat(\LL^H)=\LL$.
So, the notion of a $G$-equivariant smooth mixed Hodge module is equivalent to that of a $G$-equivariant ``good'' variation of mixed Hodge structures. Moreover, $${\rm gr}^F_p DR^G(\LL^H)$$ is in this case just the corresponding graded part coming from the usual filtered twisted de Rham complex.  Indeed,  
let $\mathcal{V}:=\LL \otimes_{\Q} \mathcal{O}_X$ be the flat bundle
with holomorphic connection $\bigtriangledown$, whose sheaf of
horizontal sections is $\LL \otimes \C$. The bundle $\VV$ comes
equipped with a decreasing (Hodge) filtration by holomorphic
sub-bundles $\FC^p$, which satisfy the
Griffiths' transversality condition $$\bigtriangledown(\FC^p)
\subset \Omega^1_X \otimes \FC^{p-1}.$$ The bundle $\VV$ becomes a
holonomic $D$-module bifiltered by
$$W_k\VV:=W_{k-n}\LL \otimes_{\Q} \mathcal{O}_X,$$
$$F_p\VV:=\FC^{-p}\VV.$$
Note that since we work with a ``good'' variation, each $F_p\VV$ underlies a unique complex algebraic vector bundle; this can be seen by using GAGA and the logarithmic de Rham complex on a suitable algebraic compactification of $X$ (compare with \cite{Sch3}[Sect.3.4]). 
The above data constitutes the smooth algebraic mixed Hodge module $\LL^H[n]$. 
In fact, as $G$ is a group of algebraic automorphisms of the pair $(X,\LL)$, all this data is compatible with the $G$-action, making $\LL^H[n]$ into a smooth $G$-equivariant mixed Hodge module, so $\LL^H[n] \in \mh^G(X)$.
By Saito's construction \cite{Sa}, $(DR(\LL^H),F_{-\centerdot })$ coincides with the usual filtered de Rham
complex $(\Omega_X^{\centerdot}(\mathcal{V}), F^{\centerdot})$ with the
filtration on the latter being induced by Griffiths' transversality, i.e.,
$$F^p \Omega_X^{\centerdot}(\mathcal{V}):= \left[ \mathcal{F}^p
\overset{\bigtriangledown}{\to} \Omega_X^1 \otimes \mathcal{F}^{p-1}
\overset{\bigtriangledown}{\to} \cdots
\overset{\bigtriangledown}{\to} \Omega_X^i \otimes \mathcal{F}^{p-i}
\overset{\bigtriangledown}{\to} \cdots \right].$$ Moreover, the $G$-action on $(X,\LL)$ makes the filtered de Rham complex and its associated graded pieces into holomorphic $G$-equivariant complexes. As in \cite{Sch3}[Ex.5.8], this yields the following result, analogous to the module property (\ref{l1}):
\begin{theorem}\label{module}
Let $G$ be a finite group of algebraic automorphisms of a complex quasi-projective manifold $X$ of pure-dimension $n$. Let $\LL$ be a $G$-equivariant ``good'' variation of mixed Hodge structures on $X$. Then:
\begin{equation}
\mc^G([\LL^H])=\chi_y(\VV) \otimes \lambda_y(T^*_X) \in K^0_G(X) \otimes
\Z[y,y^{-1}] \simeq K_0(\co^G(X)) \otimes
\Z[y,y^{-1}], 
\end{equation} 
where
$$\chi_y(\VV):=\sum_p \left[{\rm Gr}^{p}_{\mathcal{F}} \VV \right]
\cdot (-y)^{p} \in K^0_G(X)[y,y^{-1}]$$
is the $\chi_y$-characteristic of $\VV$, and $$\lambda_y(T^*_X):=\sum_p [\Lambda^pT^*_X] \cdot y^p \in K^0_G(X) \otimes
\Z[y]$$ is the total
$\lambda$-class of $T^*_X$. 
In particular, the following normalization property holds for $X$ smooth quasi-projective:
\begin{equation}\label{normalc}\mc^G([\Q_X^H])=\lambda_y(T^*_X) \in K^0_G(X) \otimes
\Z[y]=K_0(\co^G(X)) \otimes
\Z[y].
\end{equation}
\end{theorem}

\begin{remark}\rm For $X=pt$ a point space, there is an identification $\mh^G(pt) \simeq {G-}\ms^p$ of $G$-equivariant mixed Hodge modules over a point with the $G$-equivariant (graded) polarizable mixed Hodge structures, so that for $[H] \in K^G_0(\ms^p)$ we get:
$$\mc^G([H])=\chi_y([H]):=\sum_p \left[{\rm Gr}^{p}_{\mathcal{F}} H \right]
\cdot (-y)^{p} \in K^0_G(pt)[y,y^{-1}],$$ with $K^0_G(pt)$ the complex representation ring of $G$.

\end{remark}

For the proof of the multiplicativity of $\mc^G$ with respect to exterior products, a slightly more general result than the above theorem is needed. Let $X$ be a  complex quasi-projective $G$-manifold of pure-dimension $n$, with $D$ a $G$-invariant simple normal crossing divisor in $X$. Let $$j:U:=X \setminus D \hookrightarrow X$$ be the open inclusion, and $\LL$ a $G$-equivariant ``good'' variation of mixed Hodge structures on $U$. By Saito's theory \cite{Sa}, 
$${\rm gr}^F_p DR^G(j_*\LL^H)$$ 
is in this case just the corresponding graded part coming from the usual filtered twisted meromorphic de Rham complex with its induced $G$-action. Moreover, the inclusion of the $G$-equivariant twisted logarithmic de Rham complex into the latter is a  filtered quasi-isomorphism  (see also \cite{Sch3}[Thm.5.1]). This yields then the following result (compare with \cite{Sch3}[Ex.5.8]):
\begin{equation}\label{log}
\mc^G([j_*\LL^H])=\chi_y(\bar{\VV}) \otimes \lambda_y(\Omega^1_X(\log D)) \in K^0_G(X)\otimes \Z[y,y^{-1}] \:,
\end{equation} 
where $\bar{\VV}$ is the canonical Deligne extension of $\VV$, with its induced ``Hodge'' filtration $\bar{\FC}^{\centerdot}$ by algebraic sub-bundles extending the ``Hodge'' filtration of $\VV$ (by our ``goodness'' assumption).\\

We can now state the following multiplicativity property of the equivariant motivic Chern class, analogous to (\ref{multL}):
\begin{theorem}
 Let $X$ and $X'$ be algebraic quasi-projective $G$- and $G'$-varieties, respectively. Then one has a commutative diagram:
\begin{equation}\label{multC}
{\small \begin{CD}
 K_0(\mh^G(X)) \times K_0(\mh^{G'}(X')) @> \mc^G \times \mc^{G'} >>
\left(K_0(\co^G(X)) \times K_0(\co^{G'}(X'))\right)[y^{\pm 1}] \\
@V \boxtimes VV @VV \boxtimes  V\\
K_0(\mh^{G\times G'}(X\times X')) @> \mc^{G \times G'}>> K_0(\co^{G\times G'}(X\times X'))[y^{\pm 1}] .
 \end{CD}}
\end{equation}
\end{theorem}

The proof of the above theorem is a formal adaptation of that of \cite{Sch3}[Cor.5.10], by using (\ref{log}) together with the observation that for a $G$-variety $X$, the Grothendieck group $K_0(\mh^G(X))$ is generated by the classes $f_*[j_*\LL^H]$, where $f:M \to X$ is a proper $G$-morphism of quasi-projective $G$-varieties with $M$ smooth, and $\LL$ a $G$-equivariant ``good'' variation of mixed Hodge structures defined on the complement of a $G$-invariant simple normal crossing divisor in $M$ (as above). For this, one uses equivariant resolution of singularities, as in \cite{Bit}[Sect.7].\\

We end this section with a discussion on the relation between the equivariant motivic Chern class and the non-equivariant motivic Chern class for spaces with trivial $G$-action. This will be needed later on in Sect.\ref{comporb}, for computing characteristic classes of global quotients. 
Let $G$ act trivially on the quasi-projective variety $X$. Then one can 
consider the projector $$\PC_G:=\frac{1}{\vert G \vert} \sum_{g \in G} \psi_g=(-)^G$$ acting on the categories $D^{b,G}\mh(X)$ and $D^{b,G}_{\rm coh}(X)$, for $\psi_g$ the isomorphism induced from the action of $g \in G$. Here we use the fact that the underlying categories $D^{b}\mh(X)$ and $D^{b}_{\rm coh}(X)$ are $\Q$-linear additive categories which are Karoubian by \cite{BS,LC} (i.e., any projector has a kernel, see also \cite{MS09}).  
Since $\PC_G$ is exact, we obtain induced functors on the Grothendieck groups:
$$[-]^G:K_0 (\mh^G(X)) \to K_0 (\mh(X))$$ and
$$[-]^G:K_0({\co}^G(X)) \to K_0({\co}(X)).$$

We now have the following result
\begin{prop}\label{L1}
Let $X$ be a complex quasi-projective $G$-variety, with a trivial action of the finite group $G$. Then the following diagram commutes:
\begin{equation}\label{LL1}
\begin{CD}
K_0(\mh^G(X)) @> \mc^G >> K_0(\co^G(X)) \otimes \Z[y^{\pm 1}]\\
@V [-]^G VV @VV [-]^G  V\\
K_0(\mh(X)) @> \mc >> K_0(\co(X)) \otimes \Z[y^{\pm 1}]\
\end{CD}
\end{equation}
where $\mc:K_0 (\mh(X)) \to K_0({\rm Coh}(X)) \otimes \Z[y^{\pm 1}]$ is the  Brasselet-Sch\"urmann-Yokura transformation (cf. \cite{BSY}).
\end{prop}

\begin{proof} Since ${\rm gr}^F_p DR^G:D^{b,G}\mh(X) \to D^{b,G}_{\rm coh}(X)$ is an additive functor, it commutes with the projectors $(-)^G$. Therefore, the equivariant  motivic Chern
class transformation $\mc^G$ also  commutes with the projectors $[-]^G$. 
Let  $\MC \in D^{b,G}\mh(X)$ be given. The following sequence of identities yields the desired result:
\begin{equation} \left[  \mc^G([\MC]) \right]^G 
= \mc^G\left([\MC]^G\right)
= \mc([\MC]^G),
\end{equation}
where the last equality  follows since $G$ acts trivially on $\MC^G
\in D^{b}\mh(X)$.
\end{proof}

\section{Hodge-theoretic Atiyah-Singer classes of singular varieties}
\subsection{Construction. Properties} The Atiyah-Singer class ${T_y}_*(X;g)$ which will be defined in this section is an equivariant generalization of the motivic Hirzebruch class ${T_y}_*(X)$ defined by Brasselet-Sch\"urmann-Yokura \cite{BSY}, in the sense that $${T_y}_*(X)={T_y}_*(X;id) \in H_{ev}^{BM}(X;\C)[y].$$ 
In fact, the definition we give here for ${T_y}_*(X;g)$ follows closely that of \cite{BSY}.

\begin{defn} The \underbar {un-normalized Atiyah-Singer class transformation} ${\widetilde{\mt_y}}_*(g)$ is defined by composing the transformation $\mc^G$ of (\ref{d}) with  the Lefschetz-Riemann-Roch transformation $td_*(g)(-)$, i.e.,
$$\widetilde{{\mt_y}_*}(g):=td_*(g) \circ \mc^G : K_0 (\mh^G(X)) \to H_{ev}^{BM}(X^g) \otimes \C[y^{\pm 1}].$$
The \underbar {normalized Atiyah-Singer class transformation} ${\mt_y}_*(g)$ is then defined as 
\begin{equation}
{\mt_y}_*(g):=\Psi_{(1+y)}  \circ \widetilde{{\mt_y}_*}(g) :K_0 (\mh^G(X)) \to H_{ev}^{BM}(X^g) \otimes \C[y^{\pm 1},(1+y)^{-1}],
\end{equation}
with the homological Adams operation $$\Psi_{(1+y)}:H_{ev}^{BM}(X^g) \otimes \C[y^{\pm 1}] \to H_{ev}^{BM}(X^g)  \otimes \C[y^{\pm 1}, (1+y)^{-1}]$$  given by multiplication with $(1+y)^{-k}$ on $H_{2k}^{BM}(X^g) \otimes \C[y^{\pm 1}]$.
\end{defn}

\begin{remark}\label{param}\rm
Note that in the above definition  we need to invert the parameter $(1+y)$ to get the right normalization condition for ${T_y}_*(X;g)$ in Prop.\ref{normal} in case $X$ is smooth. As we will see later on (see Cor.\ref{indep}), the transformation ${\mt_y}_*(g)$ factorizes through $H_{ev}^{BM}(X^g) \otimes \C[y^{\pm 1}]$  in the case when $g$ acts trivially on $X$ (as in \cite{Sch3} in the non-equivariant context). But if the action of $g$ is non-trivial, this need not be the case. A simple example is given by 
${\mt_y}_*([j_*\Q^H_U])$
for a finite order automorphism  $g$ of a quasi-projective manifold $X$ with fixed point set 
$X^g$ a smooth hypersurface of positive dimension and $j:U:=X\setminus X^g\to X$ the inclusion of the open complement. Nevertheless, the negative powers of $(1+y)$ also disappear in many other interesting cases , e.g. in  the motivic context.
\end{remark}

Note that the transformations ${\mt_y}_*(g)$ and $\widetilde{{\mt_y}_*}(g)$ commute with proper push-downs and restrictions to open subsets. Moreover, for a subgroup $H$ of $G$ with $g \in H$, these transformations commute with the obvious restriction functors ${\rm Res}^G_H$. 
Also, by construction, ${\mt_y}_*(id)$ is the complexified version of the transformation defined in \cite{BSY,Sch3}. Finally, (\ref{multL}) and (\ref{multC}) yield the following multiplicativity property:
\begin{cor}
 Let $X$ and $X'$ be algebraic quasi-projective $G$- and $G'$-varieties, respectively. Then for $g\in G$ and $g'\in G'$ one has a commutative diagram:
\begin{equation}\label{multAS}
{\small \begin{CD}
 K_0(\mh^G(X)) \times K_0(\mh^{G'}(X')) @> \widetilde{{\mt_y}_*}(g) \times \widetilde{{\mt_y}_*}(g') >>
\left( H_{ev}^{BM}(X^g) \times H_{ev}^{BM}(X'^{g'}) \right) \otimes \C[y^{\pm 1}] \\
@V \boxtimes VV @VV \boxtimes  V\\
K_0(\mh^{G\times G'}(X\times X')) @> \widetilde{{\mt_y}_*}((g,g'))>> H_{ev}^{BM}(X^g \times X'^{g'}) \otimes \C[y^{\pm 1}].
 \end{CD}}
\end{equation}
And similarly for the transformation ${\mt_y}_*(g)$.
\end{cor}

Distinguished choices of elements in $K_0(\mh^G(X))$ yield the following characteristic homology classes: 

\begin{defn} Let $X$ be a quasi-projective variety with an algebraic action by a finite group $G$ of automorphisms. Then for each $g \in G$ we define:
\begin{enumerate}
 \item[(a)] The \underbar {(homology) Atiyah-Singer class} of $X$ is given  by:
\begin{equation}
 {T_y}_*(X;g):={T_y}_*(g)([id_X])={\mt_y}_*(g)([\Q^H_X]),
\end{equation}
for $\Q^H_X$ the constant Hodge sheaf (with its induced $G$-action as a mixed Hodge module complex).
\item[(b)] If $X$ is a manifold of pure dimension $n$ and $\LL$ a $G$-equivariant ``good'' variation of mixed Hodge structures, we define \underbar{twisted Atiyah-Singer classes} $ {T_y}_*(X,\LL;g)$ by:
\begin{equation}
 {T_y}_*(X,\LL;g):={\mt_y}_*(g)([\LL^H]),
\end{equation}
for $\LL^H$ the corresponding (shifted) smooth $G$-equivariant mixed Hodge module.
\item[(c)] Assume $X$ is pure $n$-dimensional, with $\LL$ a $G$-equivariant ``good'' variation of mixed Hodge structures on a smooth Zariski-open dense $G$-invariant subset of $X$. To the corresponding $G$-equivariant intersection homology mixed Hodge modules $IC_X^H$ resp. $IC_X^H(\LL)$ we assign the following classes:
\begin{equation}\begin{split}
{IT_y}_*(X;g)&:={\mt_y}_*(g)([IC_X^H[-n]]) \: \text{resp.} \: \\
{IT_y}_*(X,\LL;g)&:={\mt_y}_*(g)([IC_X^H(\LL)[-n]]). \end{split}
\end{equation}
\end{enumerate}
The classes obtained by using the transformation  $\widetilde{{\mt_y}_*}(g)$  will be denoted ${\widetilde{T_y}}_*(X;g)$, etc.
\end{defn}

And again, for $g=id$ the identity element of $G$, the above classes are just complexified versions of the corresponding Hirzebruch characteristic classes from \cite{BSY,CMS,MS,Sch3}. Important special properties of these classes are described as follows.

\begin{prop}\label{normal} {\rm (Normalization)} \newline If $X$ is a quasi-projective manifold acted upon by a finite group $G$ of algebraic automorphisms, the following property holds for each $g \in G$:
\begin{equation} \widetilde{{T_y}}_*(X;g)=\widetilde{T}^*_y(X;g) \cap [X^g] \ \ \ \text{and} \ \ \  {T_y}_*(X;g)= T^*_y(X;g) \cap [X^g]. \end{equation}
\end{prop}

\begin{proof} This is just a particular case of Theorem \ref{thm1} below, see Remark \ref{n}.
\end{proof}

\begin{remark}\label{normal2}\rm
For a quasi-projective manifold $X$ as in Prop.\ref{normal} we get by Remark \ref{23} the identifications:
$${T_{-1}}_*(X;g)=c^*(X^g) \cap [X^g]=c_*(1_{X^g})=c_*(g)([\Q_X]),$$
with  $\Q_X=rat (\Q^H_X)$ the constant $G$-equivariant sheaf on $X$ so that the trace of $g$ acting on stalks over $X^g$ equals the constant constructible function $1_{X^g}$. Similarly, we have:
$${T_{0}}_*(X;g)=td^*(X^g) \cap [X^g]=td_*(X^g)=td_*(g)([\OO_X]).$$
By functoriality of all of these transformations, this implies the commutativity of the diagram (\ref{comm}) from the Introduction, since by equivariant resolution of singularities, $K^G_0(var/X)$ is generated by classes $[f:M \to X]$ with $M$ a quasi-projective $G$-manifold and $f$ a proper $G$-morphism. By (\ref{normalc}), this also implies the factorizations used in (\ref{comm}):
$$\mc^G\circ \chi^G_{\rm Hdg}:K_0^G(var/X) \to K_0(\co^G(X))[y]$$
and 
$${T_y}_*(g):K_0^G(var/X) \to H^{BM}_{ev}(X^g;\C)[y],$$ where negative powers of $y$ and $(1+y)$ do not appear at all.
\end{remark}

Over a point space, both transformations ${\mt_y}_*(g)$ and $\widetilde{{\mt_y}_*}(g)$  coincide with the
equivariant $\chi_y(g)$-genus ring homomorphism $\chi_y(g):K^G_0(\ms^p) \to
\C[y,y^{-1}]$, defined on the Grothendieck group of $G$-equivariant (graded) polarizable mixed Hodge structures by 
\begin{equation} \chi_y(g)([H]):=\sum_p {\rm trace}(g \vert  Gr^p_F(H \otimes \C)) \cdot (-y)^p,
\end{equation}
for $F^{\centerdot}$ the Hodge filtration of $H \in {G-}\ms^p$. Indeed, using the identification $\mh^G(pt) \simeq {G-}\ms^p$ of $G$-equivariant mixed Hodge modules over a point with the $G$-equivariant (graded) polarizable mixed Hodge structures, we obtain for $[H] \in K^G_0(\ms^p)$:
\begin{eqnarray*} {\mt_y}_*(g)\left([H]\right) &=& {\mt_y}_*(g)\left([H]\right)\\ &=& \sum_p td_0(g)\left( [{\rm Gr}^p_F H] \right) \cdot (-y)^p\\ &=& \sum_p {\rm trace}(g \vert {\rm Gr}^p_F H) \cdot (-y)^p,
\end{eqnarray*}
where the last equality follows from the definition of the Lefschetz-Riemann-Roch transformation (compare with (\ref{l2})). In particular, for $X$ projective, by pushing-down under the constant map $k:X \to pt$, we get the following degree formula: 
\begin{prop}\label{degree} {\rm (Degree)} \newline If $X$ is a (possibly singular) complex projective variety acted upon by a finite group $G$ of algebraic automorphisms, then for any $g \in G$ we have:
\begin{equation} \chi_y(X;g)=\int_{[X^g]}\widetilde{{T_y}}_*(X;g)=\int_{[X^g]}{T_y}_*(X;g). \end{equation}
\end{prop}

Here we use implicitely that the mixed Hodge structures of Deligne and resp. M. Saito on $H^*(X;\Q)$ coincide (see \cite{Sa1}). Similarly, the degree of ${IT_y}_*(X;g)$ calculates the corresponding invariant $I\chi_y(X;g)$ coming from the Hodge filtration on the intersection cohomology $IH^*(X;\Q)$ of $X$, as studied in \cite{eCMS}. 

\subsection{Monodromy contributions}\label{mon} We next discuss  equivariant generalizations of some of the results from \cite{CLMS,CLMS2,MS}. These results  are characteristic class versions of the formulae described in \cite{eCMS}, and should also be regarded as Atiyah-Singer type theorems ``with coefficients", similar to the module property (\ref{l1}). 

\begin{theorem}\label{thm1}
Let $X$ be a  complex quasi-projective manifold of pure-dimension $n$, and $\LL$ a ``good'' variation of mixed Hodge structures on $X$. Let $G$ be a finite group of algebraic automorphisms of $(X,\LL)$. Then for any $g \in G$ we have:
\begin{equation}\label{fourb}
\widetilde{{T_y}}_*(X,\LL;g)
=\left( ch(\chi_y(\VV) \vert_{X^g})(g) \cup {\widetilde{T}_y}^*(X;g) \right) \cap [X^g],\end{equation} 
or, in its normalized form, 
\begin{equation}\label{four}
{T_y}_*(X,\LL;g)
=\left( ch_{(1+y)}(\chi_y(\VV) \vert_{X^g})(g) \cup {T_y}^*(X;g) \right) \cap [X^g],\end{equation} 
where $\chi_y(\VV)$ is the (equivariant) 
$\chi_y$-characteristic of the associated complex algebraic vector bundle $\VV:=\LL \otimes_{\Q} \mathcal{O}_X$, with its induced Hodge filtration 
$\FC^{\centerdot}$. In particular
\begin{equation}\label{mon-c}
{T_y}_*(X,\LL;g)\in H^{BM}_{ev}(X^g) \otimes \C[y^{\pm 1}],
\quad \text{with} \quad {T_{-1}}_*(X,\LL;g)=c_*(g)([\LL])\:.
\end{equation}
\end{theorem}

\begin{proof}
From Thm.\ref{module} we have that:
\begin{eqnarray*}
\mc^G([\LL^H]) = \chi_y(\VV) \otimes \lambda_y(T^*_X) \in K_0(\co^G(X)) \otimes
\Z[y,y^{-1}],
\end{eqnarray*} 
where $\lambda_y(T^*_X):=\sum_p [\Lambda^pT^*_X] \cdot y^p$ is the total
$\lambda$-class of $T^*_X$.

Since $X$ is an algebraic manifold, the equivariant Todd class transformation  $$td_*(g)(-):K_0(\co^G(X)) \to H^{BM}_{ev}(X^g;\C)$$ of the Lefschetz-Riemann-Roch theorem  is explicitly described by (see (\ref{l0})): 
\begin{equation}  td_*(g)(-)=ch(-|_{X^g})(g) \cdot td^*(X;g) \cap [X^g].
\end{equation}

By linearly extending the transformation $td_*(g)(-)$ over $\Z[y,y^{-1}]$, we can now write:
$$td_*(g)(\mc^G([\LL^H]))=ch(\chi_y(\VV)|_{X^g})(g) \cdot ch(\lambda_y(T^*_X)|_{X^g})(g)  \cdot \prod_{0<\theta <2\pi} \U_{\theta}(N^g_{\theta}) \cdot td(X^g) \cap [X^g].$$
On the other hand, by using the identity $$T^*_X \vert_{X^g}=T^*_{X^g} \oplus \sum_{0<\theta <2\pi}{N^g_{\theta}}^*,$$ the right-hand side of the above equality can be expanded as
\begin{equation}\label{eight} ch(\chi_y(\VV) \vert_{X^g})(g) \cdot ch (\lambda_y(T^*_{X^g}))(g) \cdot \prod_{0<\theta <2\pi} ch (\lambda_y({N^g_{\theta}}^*))(g) \cdot \prod_{0<\theta <2\pi} \U_{\theta}(N^g_{\theta}) \cdot td(X^g) \cap [X^g]. \end{equation}
We further note that $ch (\Lambda^p {N^g_{\theta}}^*)(g)=e^{-ip\theta} \cdot ch (\Lambda^p {N^g_{\theta}}^*)$, therefore $$ch (\lambda_y ({N^g_{\theta}}^*))(g)=\prod_j (1+ye^{-i\theta-\alpha_j}),$$ for $\{ \alpha_j \}$ the Chern roots of $N^g_{\theta}$.
Putting all of this into (\ref{eight}) yields (\ref{fourb}):
\begin{equation*}\label{nine} 
td_*(g)(\mc^G([\LL^H]))=ch(\chi_y(\VV) \vert_{X^g})(g) \cdot \widetilde{T}_y(X^g) \cdot \prod_{0<\theta <2\pi} \widetilde{T}_y^{\theta}(N^g_{\theta}) \cap [X^g].
\end{equation*}
The result as stated in (\ref{four}) is now obtained by an easy re-writing of this equation. (Similar considerations are carried out in detail, e.g., in \cite{MS}[Thm.4.1]). The final equality (\ref{mon-c}) follows from (\ref{chern}), together with
$$[\chi_{-1}(\VV)]=[\VV]\in K^0_G(X)$$
and the fact that the (locally constant) traces of $g$ on $\VV|_{X^g}$ and resp.  $\LL|_{X^g}$ agree.
\end{proof}

\begin{remark}\label{n}\rm By letting $\LL$ be the constant variation $\Q_X$ on $X$, formula (\ref{four}) specializes to the ``normalization axiom" of Proposition \ref{normal}. Indeed, the associated flat bundle is in this case (at K-theoretic level) just the structure sheaf $\mathcal{O}_X$ with its trivial Hodge filtration, i.e., ${\rm Gr}^{p}_{\mathcal{F}} \mathcal{O}_X=0$, for all $p \neq 0$.
\end{remark}

As an easy application of (\ref{log}) we also get the following result, which for simplicity we state only in the case when $g$ acts trivially on $X$ (as needed in Cor.\ref{indep} below):

\begin{prop}\label{open}
Let $X$ be a  complex quasi-projective $G$-manifold of pure-dimension $n$, with $D$ a $G$-invariant simple normal crossing divisor in $X$. 
Let $j:U:=X \setminus D \hookrightarrow X$ be the open inclusion. Consider a $G$-equivariant ``good'' variation $\LL$ of mixed Hodge structures on $U$, with  $\bar{\VV}$ the canonical Deligne extension of $\VV:=\LL \otimes \OO_U$, together with its induced ``Hodge'' filtration $\bar{\FC}^{\centerdot}$ by algebraic sub-bundles extending the ``Hodge'' filtration of $\VV$.  If $g$ acts trivially on $X$, then:
\begin{equation}\label{log2}
\widetilde{{{\mt}_y}_*}(g)([j_*\LL^H])=ch ( \chi_y(\bar{\VV}))(g)  \cdot  ch(\lambda_y(\Omega^1_X(\log D)) ) \cap td_*(X)
\end{equation} 
and
\begin{multline}\label{log3}
{{\mt}_y}_*(g)([j_*\LL^H])=ch_{(1+y)}( \chi_y(\bar{\VV}))(g) \cdot  ch_{(1+y)}( \lambda_y(\Omega^1_X(\log D)) ) \cap \Psi_{(1+y)}\left(td_*(X)\right).
\end{multline} 
\end{prop}

\begin{cor}\label{indep}  Let $X$ be a quasi-projective $G$-variety, with $g$ acting trivially  as before. Then the transformation ${\mt_y}_*(g)$ factorizes through $H_{ev}^{BM}(X) \otimes \C[y^{\pm 1}]$, i.e.,
$${\mt_y}_*(g) : K_0 (\mh^G(X)) \to H_{ev}^{BM}(X) \otimes \C[y^{\pm 1}] \subset H_{ev}^{BM}(X) \otimes \C[y^{\pm 1},(1+y)^{-1}].$$
Moreover, by substituting $y=-1$, we have the identification of transformations $${\mt_{-1}}_*(g)=c_*(g)\circ rat: K_0 (\mh^G(X)) \to H_{ev}^{BM}(X) \otimes \C.$$
\end{cor}
\begin{proof} By functoriality, it suffices to prove the statement in the context of Prop.\ref{open}, since the Grothendieck group $K_0 (\mh^G(X))$ is generated by proper direct images of classes of such $j_*\LL^H$, as follows by using (non-equivariant) resolutions of singularities.
Note that in the setup of Prop.\ref{open}, we have $$ch_{(1+y)}( \lambda_y(\Omega^1_X(\log D)) ) \cap \Psi_{(1+y)}\left(td_*(X)\right)={{\mt}_y}_*([j_*\Q_X^H]) \in H_{ev}^{BM}(X) \otimes \C[y^{\pm 1}],$$ and ${{\mt}_{-1}}_*([j_*\Q_X^H])=c_*([Rj_*\Q_X])$, 
as follows from \cite{Sch3}[Prop.5.21]. Next, we use that 
$$[\chi_{-1}(\bar{\VV})]=[{\bar \VV}]\in K^0_G(X)$$
and the fact that the (locally constant) traces of $g$ on ${\bar \VV}$ and resp. $\LL$ agree. Finally, we get:
$${\mt_{-1}}_*(g)([j_*\LL^H])=tr_g(\LL) \cdot c_*([Rj_*\Q_X])=c_*(g)([Rj_*\LL]), $$ since the transformation $tr_g$ from $D^{b,G}_c(X;\Q)$ to constructible functions on $X$ commutes with $Rj_*$, see \cite{S}[p.138].
\end{proof}

As an application of Theorem \ref{thm1} to the relative setting, we also have the following result:

\begin{theorem}\label{thm2}
Let $f:Y \to X$ be a $G$-equivariant proper morphism of quasi-projective varieties, with $X$ smooth and connected, and $G$ a finite group of algebraic automorphisms of $Y$ and resp. $X$. Assume for simplicity that 
$f$ is a  locally trivial topological fibration in the complex topology. Then, if for $g \in G$ we let $f^g:Y^g \to X^g$ denote the induced map,  we have that: 
\begin{equation}\label{fiveb}
f^g_*\widetilde{{T_y}}_*(Y;g)
=ch(\chi_y(f) \vert_{X^g})(g) \cap \widetilde{{T_y}}_*(X;g),
\end{equation}
or, in normalized form,
\begin{equation}\label{five}
f^g_*{T_y}_*(Y;g)
=ch_{(1+y)}(\chi_y(f) \vert_{X^g})(g) \cap {T_y}_*(X;g),
\end{equation}
where
$$\chi_y(f):=\sum_{i,p} (-1)^i \left[ {\rm Gr}^p_{\mathcal{F}} \HC_i
\right] \cdot (-y)^p \in K^0_G(X)[y]$$
is the $K$-theory equivariant 
$\chi_y$-characteristic of $f,$ for $\HC_i$ the flat bundle whose sheaf of horizontal
sections is  $R^if_*\C_Y$.
\end{theorem}

\begin{proof}
We have the following identity in $K_0(\mh^G(X))$:
\begin{equation}\label{i5} \left[ f_*\Q^H_Y \right]=\sum_{i \in \Z} (-1)^i [H^i(f_* \Q^H_Y)]=\sum_{i \in \Z}
(-1)^i \left[ H^{i+\text{dim}X}(f_* \Q^H_Y)[-\text{dim}X]\right].
\end{equation}
Note that $H^{i+\text{dim}X}(f_* \Q^H_Y) \in \mh^G(X)$ is the smooth
$G$-mixed Hodge module on $X$ whose underlying rational $G$-complex is
(recall that $X$ is smooth)
\begin{equation} rat (H^{i+\text{dim}X}(f_*
\Q^H_Y))={^p\HC}^{i+\text{dim}X}(Rf_*\Q_Y)=(R^i
f_*\Q_Y)[\text{dim} X],
\end{equation}
where ${^p\HC}$ denotes the perverse cohomology functor. So the local systems $\LL_i:=R^{i} f_*\Q_Y$ underly a $G$-equivariant ``good'' variation of mixed Hodge structures. By
applying the natural transformations ${\widetilde{\mt_y}_*}(g)(-)$ and ${\mt_y}_*(g)(-)$, respectively, to the equation
(\ref{i5}), and using the fact that $f$ (hence $f^g$) is proper, we obtain:
$$f^g_*\widetilde{{T_y}}_*(Y;g)=\sum_i(-1)^i \widetilde{{T_y}}_*(X,\LL_i;g)$$ and 
$$f^g_*{T_y}_*(Y;g)=\sum_i(-1)^i {T_y}_*(X,\LL_i;g),$$ respectively.
In view of Theorem \ref{thm1} this yields the desired result.
\end{proof}

\begin{remark}\rm
The above proof only uses the fact that all direct image sheaves $\LL_i:=R^{i} f_*\Q_Y$ are locally constant, thus the theorem holds under this weaker assumption. Similar formulae can be obtained for $f^g_*{IT_y}_*(Y;g)$ and $f^g_*\widetilde{{IT_y}}_*(Y;g)$, for $Y$ pure dimensional. These give a class version of the corresponding equivariant trace formulae of \cite{eCMS}, and provide  equivariant analogues of the class formulae from \cite{MS}.
\end{remark}

\begin{remark}\rm If $X$ is a (possibly singular) complex projective variety acted upon by a finite group $G$ of algebraic automorphisms and $f:X \to \{ {pt} \}$ is the contant map to a point (which is regarded as a $G$-equivariant morphism),  our formulae  $(\ref{fiveb})$ and $(\ref{five})$ specialize to the ``degree axiom" of Proposition \ref{degree}.

\end{remark}

\section{Equivariant Hirzebruch classes and invariants of global quotients}

\subsection{A computation of Hirzebruch classes of global quotients}\label{comporb} In this section we give a formula for the normalized homology Hirzebruch class  ${T_y}_*(X/G)$ of the global quotient $X/G$, for the algebraic action of a finite group $G$ on a (possibly {\it singular}) quasi-projective variety $X$. All results in this section can easily be reformulated for the corresponding un-normalized Hirzebruch classes (even with coefficients).

Our formula for ${T_y}_*(X/G)$ below is motivated by the well-known relation (e.g., see \cite{eCMS}):
\begin{equation}\label{two}
\chi_y(X/G)=\frac{1}{\vert G \vert} \sum_{g \in G} \chi_y(X;g).
\end{equation}
Indeed, if $X$ is projective, the left-hand side of (\ref{two}) is the degree of the zero-dimensional component of ${{T_y}}_*(X/G)$, while by Prop.\ref{degree} each term $\chi_y(X;g)$ in the sum of the right-hand side is the degree of the previously defined normalized Atiyah-Singer class  ${{T_y}}_*(X;g)$. Thus it is natural to expect that ${{T_y}}_*(X/G)$ is obtained by averaging (in the appropriate sense) the Atiyah-Singer classes ${{T_y}}_*(X;g)$ corresponding to various elements $g \in G$. One of the main results of this section is the following:
\begin{theorem}\label{orb} Let $G$ be a finite group acting by algebraic automorphisms on the complex quasi-projective variety $X$.  Let  $\pi^g:X^g \to X/G$ be the composition of the projection map $\pi:X \to X/G$ with the inclusion $i^g:X^g \hookrightarrow X$. Then
\begin{equation}\label{Z}
 {{T_y}}_*(X/G)=\frac{1}{\vert G \vert} \sum_{g \in G} \pi^g_*{{T_y}}_*(X;g).
\end{equation}
\end{theorem}

\begin{remark}\rm Theorem \ref{orb} is an essential ingredient for obtained generating series formulae for homology Hirzebruch classes of symmetric products of {\it singular} complex quasi-projective algebraic varieties, see \cite{CMSSY} (compare also with \cite{MS09,MSS}, where the case of Hodge numbers and Hodge-Deligne polynomials  is discussed). \end{remark}

The proof of Theorem \ref{orb} uses the following lemma in the special case $\MC=\Q^H_{X/G}$:

\begin{lemma}\label{L2}
Let $X':=X/G$ be the quotient space, with finite projection map $\pi:X \to X'$, which is viewed as a $G$-map with trivial action on $X'$. Then:
\begin{equation}
\left(\pi_* \pi^*\MC \right)^G \simeq \MC,
\end{equation} for any $\MC \in D^b\mh(X')$, 
where $\pi^*\MC$ and resp. $\pi_* \pi^*\MC$ get an induced $G$-action as explained in the Appendix. Here, 
$$(-)^G:=\frac{1}{\vert G \vert} \sum_{g \in G} \psi_g : D^{b,G}\mh(X') \to D^{b}\mh(X')$$
denotes the projector onto the $G$-invariant sub-object.
\end{lemma}

\begin{proof} Consider the composite morphism
\begin{equation}\label{c}
\MC \to \pi_* \pi^*\MC \to (\pi_* \pi^*\MC)^G,
\end{equation}
with the first arrow defined by the adjunction morphism $id \to \pi_*\pi^*$, and the second being induced by the projector.

The claim is that the morphism (\ref{c}) is in fact an isomorphism. This is equivalent to showing that  the cone of (\ref{c}) is $0\in D^b\mh(X')$. And the latter statement can be checked  after applying the forgetful functor ${\rm rat}: D^b\mh(X') \to D^b_c(X';\Q)$, which commutes with the projector onto the $G$-invariant part, as well as with the adjunction morphism.
Note that one has the following sequence of equivalences for a mixed Hodge module complex on the complex algebraic variety $X'$:
\begin{equation}\label{rat} \begin{split}
\MC \simeq 0 \in D^b\mh(X') &\Leftrightarrow H^i(\MC)\simeq 0 \in \mh(X'), \:\text{for all $i$}\\
&\Leftrightarrow  rat\left(  H^i(\MC)\right) = {^p\HC}^i(rat(\MC)) \simeq 0 \in Perv(X';\Q), \:\text{for all $i$}\\
&\Leftrightarrow rat(\MC)\simeq 0 \in D^b_c(X';\Q) \:. \end{split}
\end{equation}
The first equivalence above follows from the definition of a derived category, the second uses the faithfulness of the forgetful functor $rat: \mh(X')\to Perv(X';\Q)$ associating to a mixed Hodge module
the underlying perverse sheaf (with ${^p\HC}^*(-)$ denoting the \emph{perverse cohomology groups}), while the last equivalence follows from the boundedness of the perverse $t$-structure on $ D^b_c(X';\Q)$. Lastly, for $\FC\in D^b_c(X';\Q)$, the underlying sheaf map $ \FC \to (\pi_*\pi^*\FC)^G$ corresponds by the ($G$-equivariant) projection formula   
$$\pi_*\pi^*\FC \simeq  \pi_*(\pi^* \FC\otimes \pi^*\Q_{X'}) \simeq \FC \otimes (\pi_*\pi^*\Q_{X'})$$ to the morphism
$$\FC \otimes \Q_{X'} \to \FC \otimes (\pi_*\pi^*\Q_{X'})^G$$ induced by the adjunction $\Q_{X'} \to  (\pi_*\pi^*\Q_{X'})^G$, which is an isomorphism by \cite{G}[ch.V]. 
\end{proof}

\begin{proof}[Proof of Thm.\ref{orb}] It suffices to prove the un-normalized version of the result. 
First recall that it follows from  the Lefschetz-Riemann-Roch theorem that the following identity holds for any $\FC \in K_0({\rm Coh}^G(X))$, e.g., see \cite{BFQ,M}:
\begin{equation}\label{id3}
td_*((\pi_*\FC)^G)=\frac{1}{\vert G \vert} \sum_{g \in G}  \pi^g_* td_*(g)(\FC).
\end{equation}
Theorem \ref{orb} follows now from the following sequence of identities: 
 {\allowdisplaybreaks
\begin{eqnarray*}  \widetilde{T_y}_*(X')
&\overset{\rm def}{=}& \widetilde{{\mt_y}_*}([\Q^H_{X'}]) \\
&\overset{{\rm Lem.}  \ref{L2}}{=}& td_* \left(  \mc([\pi_* \pi^*\Q^H_{X'}]^G) \right)\\
&=& td_* \left(  \mc([\pi_* \Q^H_X]^G) \right) \\
&\overset{{\rm Prop.}  \ref{L1}}{=}& td_* \left(  \left[\mc^G([\pi_* \Q^H_X])\right]^G \right) \\
&=& td_* \left(  \left[\pi_*(\mc^G([\Q^H_X]))\right]^G \right)\\
&\overset{(\ref{id3})}=&  \frac{1}{\vert G \vert} \sum_{g \in G}  \pi^g_* td_*(g)\left(\mc^G([\Q^H_X])\right) \\
&=& \frac{1}{\vert G \vert} \sum_{g \in G}  \pi^g_*\widetilde{T_y}_*(X;g).
\end{eqnarray*} }
\end{proof}

The above argument yields in fact the following more general result: 
\begin{theorem}\label{gen} Let $G$ be a finite group acting by algebraic automorphisms on the complex quasi-projective variety $X$.  Let  $\pi^g:X^g \to X/G$ be the composition of the projection map $\pi:X \to X/G$ with the inclusion $i^g:X^g \hookrightarrow X$. Then, for any $\MC \in D^{b,G}\mh(X)$, we have:
\begin{equation}\label{arbn}
{{\mt}_y}_*([\pi_*\MC]^G)=\frac{1}{\vert G \vert} \sum_{g \in G} \pi^g_* {{\mt}_y}_*(g)([\MC]).
\end{equation}
\end{theorem}

Combining this with Lemma \ref{L2}, we get the following result which allows to calculate invariants on $X/G$ in terms of equivariant invariants on $X$:
\begin{cor}\label{cornew} Let $G$ be a finite group acting by algebraic automorphisms on the complex quasi-projective variety $X$.  Let  $\pi^g:X^g \to X/G$ be the composition of the projection map $\pi:X \to X/G$ with the inclusion $i^g:X^g \hookrightarrow X$. Then, for any $\MC \in D^{b}\mh(X/G)$, we have:
\begin{equation}\label{arbn2}
{{\mt_y}_*}([\MC])=\frac{1}{\vert G \vert} \sum_{g \in G} \pi^g_* {\mt_y}_*(g)([\pi^*\MC]).
\end{equation}
\end{cor}

Another interesting example for using Theorem \ref{gen} comes from intersection cohomology (with twisted coefficients):
\begin{lemma}\label{L3} 
Let $X$ be a pure-dimensional complex quasi-projective variety acted upon by a finite group $G$ of algebraic automorphisms. Let $X':=X/G$ be the quotient space, with finite projection map $\pi:X \to X'$, which is viewed as a $G$-map with trivial action on $X'$. Let $\LL$ be a $G$-equivariant ``good" variation of mixed Hodge structures on a $G$-invariant Zariski open dense smooth subset $\U \subset X$ with $\pi':=\pi_{\vert \U} :\U \to \U':=\U/G$ an unramified finite covering of algebraic manifolds. 
Then $\pi'_*\LL$ is a $G$-equivariant ``good" variation of mixed Hodge structures on $\U'$ and
\begin{equation}\label{LL3}
(\pi_*IC^H_X(\LL))^G\simeq IC^H_{X'}((\pi'_*\LL)^G).
\end{equation}
In particular, 
\begin{equation}\label{LL3b}
(\pi_*IC^H_X(\pi'^*\LL'))^G\simeq IC^H_{X'}(\LL')
\end{equation}
for $\LL'$ a ``good" variation of mixed Hodge structures on $\U'$. For example, if $\LL'=\Q^H_{\U'}$ one obtains: 
$$(\pi_*IC^H_X)^G\simeq IC^H_{X'}.$$
\end{lemma}

\begin{proof} Since $\pi'$ is an unramified covering, $\pi'_*\LL$ is a $G$-equivariant local system on $\U'$. Since $\pi'_*\LL=rat(\pi'_*\LL^H)$ underlies a $G$-equivariant smooth mixed Hodge module, it is in fact a $G$-equivariant ``good" variation of mixed Hodge structures. Similar considerations apply for the direct summand $(\pi'_*\LL)^G$.

By shrinking $\U$ if necessary, we can assume $\U$ (and thus $\U'$) is affine.  If $j : \U \hookrightarrow X$ denotes the inclusion map, the functors $$j_*, j_!: \mh(U) \to \mh(X)$$ are exact, and similarly for the corresponding equivariant functors induced on the abelian categories $\mh^G(-)$ of $G$-equivariant objects. Denote by $j':\U'=\U/G \to X'=X/G$ the induced open embedding. Then $j' \circ \pi'=\pi \circ j$. We also let the symbol $``!*"$ stand for the {\it intermediate extension} for an open embedding, which in our setting can be formally defined by the rule: ``$!*={\rm image}(! \to *)$".  Then the following identity holds in $\mh(X')$ and $\mh^G(X')$, respectively:
\begin{equation}\label{mainid}
j'_{!*}\pi'_*(\LL^H[n])=\pi_*j_{!*}(\LL^H[n]),
\end{equation}
for $n={\rm dim}_{\C}X$.
In fact,  the defining sequence for the intermediate extension $j_{!*}(\LL^H[n])$, i.e., the sequence 
\begin{equation}\label{seqdef} 
j_{!}(\LL^H[n]) \twoheadrightarrow j_{!*}(\LL^H[n]) \hookrightarrow j_{*}(\LL^H[n])
\end{equation}
in the abelian category  $\mh(X)$, also yields a $G$-action on $IC_X^H(\LL):= j_{!*}(\LL^H[n])$, thus defining $IC^H_X(\LL)$ as an object in the abelian category $\mh^G(X)$. This follows easily from the considerations in the Appendix, where equivariant counterparts for the exact functors $j_!$ and $j_*$ are constructed. In what follows, we omit the upperscript $^G$ for the equivariant functors, whenever it is clear what context we work in.

From here on, the proofs of (\ref{mainid}) in both cases $\mh(-)$ and $\mh^G(-)$ run in parallel.
By applying the exact functor $\pi_!=\pi_*$ for the finite (hence proper) map $\pi:X \to X'$ to the sequence of (\ref{seqdef}), we obtain the sequence
\begin{equation}\label{seqdef2} 
\pi_!j_{!}(\LL^H[n]) \twoheadrightarrow \pi_!j_{!*}(\LL^H[n]) \hookrightarrow \pi_*j_{*}(\LL^H[n])
\end{equation} in the abelian category $\mh(X')$ and $\mh^G(X')$, respectively. (Note that since $\pi$ is a finite morphism, the functor $\pi_*=\pi_!$ preserves (equivariant) mixed Hodge modules; indeed, after applying the faithful forgetful functor ${\it{rat}}$ to the category of (equivariant) perverse sheaves, this functor preserves (equivariant) perverse sheaves, i.e., it is $t$-exact with respect to the ($G$-invariant) perverse $t$-structure.)  Lastly, using the identities $\pi_* j_*=j'_* \pi'_*$ and $\pi_! j_!=j'_!  \pi'_!$, with $\pi'_!=\pi'_*$ for the finite map $\pi'$, we easily 
obtain (\ref{mainid}).

Since in the abelian category $\mh^G(-)$ the projector $\PC_G$  is also exact in short exact sequences of $G$-equivariant mixed Hodge modules, it commutes with the functor $ j'_{!*}$, and  it follows from the proof of Lemma \ref{L2} and the identity (\ref{mainid})  that:
 {\allowdisplaybreaks
\begin{eqnarray*}   {IC}_{X'}^H((\pi'_*\LL)^G)
&\overset{\rm def}{=}&  j'_{!*}(( \pi'_*\LL^H)^G[n])\\ 
&\cong& \left(  j'_{!*} \pi'_*(\LL^H[n]) \right)^G\\ 
&\overset{(\ref{mainid})}{\cong}& \left( \pi_*j_{!*}(\LL^H[n]) \right)^G\\
&\cong& (\pi_* IC_X^H(\LL))^G.
\end{eqnarray*} }
\end{proof}

The identifications of Lemma \ref{L3} can be used in the statement of Theorem \ref{gen} to obtain the following calculation of homology characteristic classes defined via (twisted) intersection homology Hodge modules:
\begin{cor}\label{twist} In the context and with the notations of Lemma \ref{L3}, the following identities hold:
 \begin{equation}\label{ZZ3}
{{IT_y}}_*(X/G,(\pi'_*\LL)^G)=\frac{1}{\vert G \vert} \sum_{g \in G} \pi^g_*{{IT_y}}_*(X,\LL;g),
\end{equation}
 \begin{equation}\label{ZZ4}
{{IT_y}}_*(X/G,\LL')=\frac{1}{\vert G \vert} \sum_{g \in G} \pi^g_*{{IT_y}}_*(X,\pi'^*\LL';g),
\end{equation}
and, in particular,
 \begin{equation}\label{ZZ2}
{{IT_y}}_*(X/G)=\frac{1}{\vert G \vert} \sum_{g \in G} \pi^g_*{{IT_y}}_*(X;g).
\end{equation}
\end{cor}

\subsection{Atiyah-Meyer formulae for global orbifolds}\label{orbifoldAM}
As an application of the results in the previous section, we obtain the following  Atiyah-Meyer type result for twisted Hirzebruch classes of global orbifolds, as already mentioned in the Introduction, see Thm.\ref{thintro4}:
\begin{theorem}\label{ASorb}
Let $G$ be a finite group of algebraic automorphisms of the pure $n$-dimensional quasi-projective manifold $X$, with $\pi:X \to X':=X/G$ the projection map. Let $\LL'$ be a local system on $X'$, which generically underlies a good variation of mixed Hodge structures defined on a Zariski dense open smooth subset $\U' \subset X'$, and let $\VV':=\LL' \otimes_{\Q} \mathcal{O}_{X'}$ be the associated complex algebraic vector bundle on $X'$. Then:
\begin{equation}\label{fiveintrob}
{IT_y}_*(X',\LL')
=ch_{(1+y)}(\chi_y(\VV') ) \cap {IT_y}_*(X') ,\end{equation} 
where $ch_{(1+y)}(\chi_y(\VV') ) \in H^{ev}(X';\Q[y^{\pm 1}])$ corresponds, by definition,  to $ch_{(1+y)}(\chi_y(\VV) ) \in H^{ev}(X;\Q[y^
{\pm 1}])^G$ under the isomorphism $$\pi^*:H^{ev}(X';\Q[y^{\pm 1}]) \overset{\sim}{\to} H^{ev}(X;\Q[y^{\pm 1}])^G \subset H^{ev}(X;\Q[y^{\pm1}]).$$ Here, $\LL:=\pi^*(\LL')$ underlies a good variation of mixed Hodge structures on all of $X$, and $\chi_y(\VV)$ denotes the $\chi_y$-characteristic of the associated complex algebraic vector bundle $\VV:=\LL \otimes_{\Q} \mathcal{O}_X$, with its induced Hodge filtration $\FC^{\centerdot}$.
\end{theorem}
\begin{proof} Let  $\pi^g:X^g \to X'$ denote as before the composition of the projection map $\pi:X \to X'$ with the inclusion $i^g:X^g \hookrightarrow X$. Without loss of generality, we can assume (by further restriction) that  $\pi':=\pi\vert_{\pi^{-1}(\U')}$ is a finite unramified covering of algebraic manifolds, so that we can directly apply Cor.\ref{twist} in our setup.
Formula (\ref{fiveintrob}) follows from the following string of equalities:
 {\allowdisplaybreaks
\begin{eqnarray*}   {IT_y}_*(X',\LL')
&\overset{(\ref{ZZ4})}{=}& \frac{1}{\vert G \vert} \sum_{g \in G} \pi^g_*{{IT_y}}_*(X,\pi'^*(\LL'\vert_{\U'});g)\\ 
&\overset{(\ast)}{=}& \frac{1}{\vert G \vert} \sum_{g \in G} \pi^g_*{{T_y}}_*(X,\LL;g)\\ 
&\overset{(\ref{four})}{=}&  \frac{1}{\vert G \vert} \sum_{g \in G} \pi^g_* \left( ch_{(1+y)}(\chi_y(\VV) \vert_{X^g})(g) \cap {T_y}_*(X;g) \right) \\ 
&\overset{(\ast\ast)}{=}&  \frac{1}{\vert G \vert} \sum_{g \in G} \pi^g_* \left( ch_{(1+y)}(\chi_y(\VV) \vert_{X^g}) \cap {T_y}_*(X;g) \right) \\ 
&=&  \frac{1}{\vert G \vert} \sum_{g \in G} \pi^g_* \left( {\pi^g}^* ch_{(1+y)}(\chi_y(\VV') ) \cap {T_y}_*(X;g) \right) \\ 
&\overset{(\ast\ast\ast)}{=}& ch_{(1+y)}(\chi_y(\VV') ) \cap \left(  \frac{1}{\vert G \vert} \sum_{g \in G} \pi^g_*{T_y}_*(X;g)  \right) \\
&\overset{(\ref{Z})}{=}& ch_{(1+y)}(\chi_y(\VV') ) \cap {T_y}_*(X')\\
&=& ch_{(1+y)}(\chi_y(\VV') ) \cap {IT_y}_*(X').
\end{eqnarray*} }

\noindent The last equality follows from the fact that $X'$ is a rational homology manifold, so that $IC_{X'}^H[-n]\cong \Q^H_{X'}$, where $n={\rm dim}_{\C}(X')$. Let us now explain the equalities labelled by $(\ast)$, $(\ast\ast)$ and $(\ast\ast\ast)$. First note that, since $X$ is smooth of dimension $n$,  
$$\LL \cong IC_X(\pi^*\LL')[-n].$$ Therefore, as $\LL'$ is generically a ``good" variation of mixed Hodge structures, it follows that $\LL$ underlies a (shifted) $G$-equivariant smooth mixed Hodge module on $X$. So $\LL$  is a $G$-equivariant ``good" variation of mixed Hodge structures on all of $X$. This explains $(\ast)$, and the fact that $$ch_{(1+y)}(\chi_y(\VV) ) \in H^{ev}(X;\Q[y^
{\pm 1}])^G.$$ 
For $(\ast\ast)$ it suffices to show that $g$ acts trivially on $\VV \vert_{X^g}$ along each connected component of $X^g$. Indeed, since $g$ acts as an automorphism of a variation of mixed Hodge structures, this implies then that $g$ also acts trivially on the corresponding pieces of the Hodge filtration of $\VV\vert_{X^g}$, as well as on the graded pieces $Gr^p_{\FC} \VV \vert_{X^g}$ appearing in the definition of $ch_{(1+y)}(\chi_y(\VV) \vert_{X^g})(g)$. But $\VV \vert_{X^g}=(\pi^g)^*(\VV')$, with the $g$-action induced by pull-back of the trivial $g$-action on $\VV'$ (as defined in the Appendix). However, $g$ acts trivially on $X'$, as well as on $X^g$, so that this induced action is also trivial. Finally, $(\ast\ast\ast)$ follows by the projection formula.
\end{proof}

\begin{remark}\label{defect}\rm Let $\LL'$ be a local system defined on a Zariski-dense open subset $\U$ of $X'$, which generically underlies a good variation of mixed Hodge structures on a Zariski-dense open smooth subset $\U' \subset \U$.
Then all arguments in the proof of the above theorem go through under the weaker assumption that only $\pi_{\U}^*\LL'$ extends to a local system $\LL$ on all of $X$, with $\pi_{\U}:=\pi \vert_{\pi^{-1}(\U)}$, except for the equality labelled by $(\ast\ast)$. Even under this weaker assumption, $ch_{(1+y)}(\chi_y(\VV') )$ can be defined as before, but the Atiyah-Meyer formula (\ref{fiveintrob}) need not be true. In fact, our proof above provides the following ``defect formula": 
\begin{multline}\label{def2}
{IT_y}_*(X',\LL')-ch_{(1+y)}(\chi_y(\VV') ) \cap {IT_y}_*(X')\\
= \frac{1}{\vert G \vert} \sum_{g \in G} \pi^g_* \left( \left( ch_{(1+y)}(\chi_y(\VV) \vert_{X^g})(g)-ch_{(1+y)}(\chi_y(\VV) \vert_{X^g}) \right) \cap {T_y}_*(X;g) \right). 
\end{multline}
Note that for a connected component $S$ of some fixed-point manifold $X^g$, the difference term
$$ch_{(1+y)}(\chi_y(\VV) \vert_{X^g})(g)-ch_{(1+y)}(\chi_y(\VV) \vert_{X^g})$$ vanishes if $g=id_G$ or if $\pi^g(S)$ is not contained in the ``singular locus" $D':=X' \setminus \U$ of $\LL'$. In fact, if $S \cap \pi^{-1}(\U) \neq \emptyset$,  then $g$ acts trivially on $\VV\vert_{S \cap \pi^{-1}(\U)}$ by the same argument as in the proof of the above theorem. Since $S \cap \pi^{-1}(\U)$ is open and dense in $S$, and $\VV$ is locally free on $S$, it follows that $g$ acts trivially on $\VV\vert_{S}$. As a consequence, the difference term appearing on the right hand side of the defect formula (\ref{def2}) can be regarded as a Borel-Moore homology class localized at the ``singular locus" $D'$ of $\LL'$. In particular, the Borel-Moore homology classes ${IT_y}_*(X',\LL')$ and $ch_{(1+y)}(\chi_y(\VV') ) \cap {IT_y}_*(X')$ agree in homology degrees greater than the complex dimension of $D'$. As an example, if $D'$ is a finite set of points in $X'$, the defect formula becomes:
\begin{multline}\label{def3}
{IT_y}_*(X',\LL')-ch_{(1+y)}(\chi_y(\VV') ) \cap {IT_y}_*(X')\\
= \frac{1}{\vert G \vert} \sum_{g \neq id_G} \left( \sum_{x \in D^g} 
\left( \chi_y(\LL_x;g) - \chi_y(\LL_x) \right) \cdot \prod_{0<\theta_x<2\pi} \frac{1+ye^{-i\theta_x}}{1-e^{-i\theta_x}}  \right), 
\end{multline}
with $D^g$ the set of isolated fixed points of $g$ which are contained in $\pi^{-1}(D')$, and $\theta_x$ the eigenvalues of $g$ acting on the tangent space $T_xX$ of $X$ at $x$.
\end{remark}

\subsection{Equivariant Hirzebruch classes}\label{EH} The main results of Section \ref{comporb} suggests that one should collect the Atiyah-Singer class transformations ${\mt_y}_*(g)$ (respectively,  ${\widetilde{{\mt_y}_*}}(g)$) all at once as a transformation 
$${\mt_y^G}_*:= \oplus_{g \in G} \ \frac{1}{\vert G \vert} \cdot {\mt_y}(g): K_0(\mh^G(X)) \to \oplus_{g \in G} \ H^{BM}_{ev}(X^g;\Q) \otimes \C[y^{\pm 1},(1+y)^{-1}].$$
Note that the disjoint union $\bigsqcup_{g \in G} X^g$ gets by $h: X^g \to X^{hgh^{-1}}$ an induced $G$-action, such that the canonical map $$i: \bigsqcup_{g \in G} X^g \to X$$ defined by the inclusions of fixed point sets is $G$-equivariant. Therefore, $G$ acts in a natural way on  $\oplus_{g \in G} \ H^{BM}_{ev}(X^g;\Q)$, and we call the $G$-invariant part $$H^{BM}_{ev,G}(X;\Q):=\left( \oplus_{g \in G} \ H^{BM}_{ev}(X^g;\Q) \right)^G$$ the {\it delocalized} $G$-equivariant Borel-Moore homology of $X$. This notion is different (except for free actions) than the equivariant Borel-Moore homology $H^G_{BM,ev}(X;\Q)$ defined by the Borel construction (e.g., see \cite{EG}[Sect.2.8]). Since $G$ is finite, and we work with $\Q$-coefficients, one just has $$H^G_{BM,ev}(X;\Q)\simeq \left(H^{BM}_{ev}(X;\Q)  \right)^G \simeq H^{BM}_{ev}(X/G;\Q),$$ which is a direct summand of $H^{BM}_{ev,G}(X;\Q)$ corresponding to the identity element of $G$.

By functoriality for $h:(X,X^g) \to (X,X^{hgh^{-1}})$, with $g,h \in G$, one has that
\begin{equation}\label{conjugation}
 h_*\left( {\mt_y}_*(g)([\MC])\right)= {\mt_y}_*(hgh^{-1})\left( h_*[\MC]\right)
= {\mt_y}_*(hgh^{-1})\left( [\MC]\right) ,
\end{equation}
where the last equality uses the isomorphism $\psi_h$ coming from the $G$-action on $\MC\in D^{b,G}(\mh(X))$.
Note that $h$ is equivariant with respect to the action of the cyclic group generated by
$g$ resp. $g':=hgh^{-1}$. Moreover, $h_*\MC$ gets by {\it conjugation} an induced action $\psi^{conj}_{g'}: h_*\MC\to g'_*h_*\MC$ of the cyclic group generated by $g'$
making the left square of the following diagram commutative:
\begin{equation}\begin{CD}
\MC @>\psi_h > \sim > h_*\MC @> h_*(\psi_g) > \sim > h_*g_*\MC \\
@V \psi_{g'} V \wr V @V \psi^{conj}_{g'} V \wr V @VV \wr V \\
g'_*\MC @>g'_*(\psi_h) > \sim > g'_*h_*\MC @>> \sim >  (h\circ g)_*\MC \:. 
\end{CD}\end{equation}
So the isomorphism $\psi_h: \MC\to h_*\MC$ is equivariant for the induced action of $g'$ on $\MC$ and $h_*\MC$, respectively,
coming from the restriction of the $G$-action and by the action induced by conjugation, respectively. This implies the last equality in (\ref{conjugation}).
But also the outer square is commutative, since
$$h_*(\psi_g) \circ \psi_h= \psi_{hg} = g'_*(\psi_h)\circ \psi_{g'}: \MC\to (h\circ g)_*\MC \:.$$
So the isomorphism $\psi^{conj}_{g'}$ makes the  right square above commutative. By definition this is also true for the isomorphism
$$\psi_{g'}:=h_*(\psi_g): h_*\MC \to h_*g_*\MC \simeq   (h\circ g)_*\MC  \simeq g'_*h_*\MC$$
induced by pushing down along $h$ (see Appendix A), giving the first equality in (\ref{conjugation}).
Hence, we just showed abstractly that these two induced actions (by conjugation and push down) of (the cyclic group generated by) $g'$ on $h_*\MC$ agree.
The same argument also applies to the Lefschetz-Riemann-Roch transformation $td_*(g)$ of \cite{BFQ, M}. \\

So the image of ${\mt_y^G}_*$ is contained in $H^{BM}_{ev,G}(X;\Q)\otimes \C[y^{\pm 1},(1+y)^{-1}]$ (see also \cite{M}[Sect.II.1] for a direct argument in case $X$ is smooth). This leads to the following 
\begin{defn}
The \underline{$G$-equivariant Hirzebruch class transformation} 
$${\mt_y^G}_*: K_0(\mh^G(X)) \to  H^{BM}_{ev,G}(X;\Q) \otimes \C[y^{\pm 1},(1+y)^{-1}]$$
is defined by: $${\mt_y^G}_*:=\oplus_{g \in G} \ \frac{1}{\vert G \vert} \cdot {\mt_y}_*(g).$$
\end{defn}

This transformation has the same properties as the Atiyah-Singer class transformations, e.g., functoriality for proper push-downs and open restrictions, restrictions to subgroups, and multiplicativity for exterior products.

Theorem \ref{gen} can now be reformulated as the commutativity of the following diagram (which proves Thm.\ref{thintro2} from the Introduction): 
\begin{equation}\label{refor}\begin{CD}
K_0(\mh^G(X)) @> {\mt_y^G}_* >> H^{BM}_{ev,G}(X;\Q) \otimes \C[y^{\pm 1},(1+y)^{-1}]\\
@V [-]^G\circ \pi_* VV @VV (\pi\circ i)_* V \\
K_0(\mh(X/G)) @> {\mt_y}_* >> H^{BM}_{ev}(X/G;\Q) \otimes \C[y^{\pm 1},(1+y)^{-1}].
\end{CD}\end{equation}

Identical considerations apply to the un-normalized Atiyah-Singer transformations and the associated un-normalized $G$-equivariant Hirzebruch transformation ${\widetilde{\mt_y^G}}_*$.
This transformation can also be defined as ${\widetilde{\mt_y^G}}_*:= td_*^G\circ \mc^G$, with
\begin{equation}
td_*^G:=\oplus_{g \in G} \ \frac{1}{\vert G \vert} \cdot td_*(g):  K_0({\rm Coh}^G(X)) \to  H^{BM}_{ev,G}(X;\Q) \:.
\end{equation}

\section{Appendix A: Equivariant categories and Grothendieck groups}

In this appendix, we collect in abstract form the results on  equivariant categories and Grothendieck groups, as needed in this paper. We follow closely the treatment from \cite{MS09}[Appendix]. For simplicity we only work in the underlying category $var^{qp}/\C$ of complex quasi-projective algebraic varieties
(whereas \cite{MS09}[Appendix] applies to any (small) category $space$ of spaces with finite products and a terminal object). 

\subsection{Equivariant categories}
Let $A$ be a {\em covariant pseudofunctor} on $var^{qp}/\C$, taking values in 
($\Q$-linear) additive (or abelian, resp. triangulated) categories. This means a ($\Q$-linear) additive (or abelian, resp. triangulated) category $A(X)$ for each $X\in ob(var^{qp}/\C)$, together with ($\Q$-linear) additive (and exact) push-down functors $f_*: A(X)\to A(Y)$ for a suitable class of morphisms
$f: X\to Y$ in $var^{qp}/\C$, which is closed under composition, exterior products and which contains all isomorphisms.
In addition, one asks for some
 natural isomorphisms $c: (kf)_*\simeq k_*f_*$ and $e: id\simeq id_*$ satisfying suitable conditions
(as spelled out in \cite{MS09}[Appendix] in a slightly more general form) to hold. Just to simplify the notation, we usually write them here as equalities
$(kf)_*= k_*f_*$ and $id=id_*$.

\begin{example}\label{pseudo}
The examples we need in this paper are the following ones:
\begin{enumerate}
 \item[(i)] $A(X)$ is the $\Q$-linear abelian category: 
\begin{enumerate}
\item[(a)] $\mh(X)$ of mixed Hodge modules;
\item[(b)] $\co(X)$ of coherent algebraic sheaves; 
\item[(c)] $Perv(X;\Q)$ of perverse sheaves with rational coefficients on $X$. 
\end{enumerate}
These are all covariant functorial for the class of {\em finite} morphisms.
\item[(ii)] $A(X)$ is the $\Q$-linear bounded derived category: 
\begin{enumerate}
\item[(a)] $D^b\mh(X)$ of mixed Hodge modules;
\item[(b)] $D^b_{\rm coh}(X)$ of sheaves of $\OO_X$-modules with coherent cohomology sheaves;  
\item[(c)] $D^b_c(X;\Q)$ of algebraically constructible sheaf complexes with rational coefficients on $X$. 
\end{enumerate}
These are all covariant functorial for the class of {\em proper} morphisms. More generally, $D^b\mh(X)$ and $D^b_c(X;\Q)$ are covariant functorial for the class of all morphisms with respect to $f_!$ resp. $f_*$. 
\end{enumerate}
\end{example}

\begin{defn}
Let $G$ be a finite group of algebraic automorphisms of $X$. Then a $G$-equivariant object $M$ in
$A(X)$ is given by the underlying object $M$, together with isomorphisms
$$\psi_g: M\to g_*M \quad (g\in G)$$
such that $\psi_{id}=id$ and $\psi_{gh}=g_*(\psi_h)\circ \psi_g$ for all $g,h \in G$.
A $G$-equivariant morphism $M\to M'$ of such equivariant objects is just a morphism in
$A(X)$ commuting with the isomorphisms $\psi_g$ above. The corresponding category $A^G(X)$ of $G$-equivariant objects is then again a ($\Q$-linear) additive 
(resp. abelian) category. Moreover, for a subgroup
$H\subset G$ one has obvious ($\Q$-linear) additive (and exact) restriction functors $Res^G_H: A^G(X)\to A^H(X)$, thus a forgetful
functor ${\text For}=Res^G_{\{id\}}: A^G(X)\to A(X)$.
\end{defn}

\begin{remark}\rm
The structure of a pseudofunctor $A(-)$ on $var^{qp}/\C$ allows one to define the category $A^{op}/(var^{qp}/\C)$ of pairs $(X,M)$, with $X\in ob(var^{qp}/\C)$ and
$M\in A(X)$, and a morphism $(X,M)\to (X',M')$ given by a (suitable) morphism $f: X\to X'$ in $var^{qp}/\C$ together with a morphism $\psi: M'\to f_*M$ in $A(X')$ (with the obvious composition). Then a $G$-equivariant object $M\in A(X)$ corresponds to a usual $G$-action on $(X,M)$ in the category  $A^{op}/(var^{qp}/\C)$ (see \cite{MS09}[Appendix]). 
So if $G$ acts {\em trivially} on $X$, this just means a $G$-action on $M$ in the categroy $A(X)$.
\end{remark}

It follows that a (suitable) $G$-equivariant morphism $f: X\to Y$ in $var^{qp}/\C$ induces a {\em functorial push-down 
functor} $f_*^G: A^G(X)\to A^G(Y)$. Indeed,  for $M\in ob( A^G(X))$, $f_*M$  inherits a $G$-action via
$$\begin{CD}
f_*M @> f_*(\psi_g) >>   f_*(g_*M)\simeq (fg)_*M=(gf)_*M\simeq g_*(f_*M).
  \end{CD}$$

For the categories mentioned in Example \ref{pseudo} we also get suitable induced {\em pullback} and  {\em exterior product functors} on the corresponding  equivariant categories. Let us first explain the pullback 
\begin{equation}f^*_G: A^G(X)\to A^G(Y) , \end{equation}
for a suitable class of $G$-equivariant morphisms $f: Y\to X$. In order to get an induced $G$-action on $f^*M$ (for $M\in ob( A^G(X))$)  we  need functorial base change isomorphisms $f^*g_*\simeq g_*f^*$ for the cartesian diagram (with $g\in G$):
 
$$\begin{CD}
  Y @> g >> Y\\
@V f VV @VV f V\\
X @> g>> X.
  \end{CD}$$ Then the $G$-action is obtained via
$$\begin{CD}
f^*M @> f^*(\psi_g) >>   f^*(g_*M)\simeq g_*f^*M.
  \end{CD}$$  

\begin{example}\label{pull}
Such base change isomorphisms are available in the following cases:
\begin{enumerate}
 \item[(i)] For all pullbacks $f^*$ (or extraordinary pullbacks $f^!$) in the cases $A(X)=D^b\mh(X)$ and $A(X)=D^b_c(X;\Q)$.
\item[(ii)] For all pullbacks $f^*: D^b_{coh}(X)\to D^b_{coh}(Y)$ under smooth (or flat)  morphisms $f: Y\to X$,
e.g., for open inclusions.
\end{enumerate}
Therefore from the constant (equivariant) elements $\Q^H_{pt}\in \mh(\{pt\}), \Q_{pt}\in Perv(\{pt\};\Q)$ and $\C_{pt}\in \co(\{pt\})$, with a trivial action of $G$,
we get by the pullback $k^*$ under a constant (flat) map $k: X\to \{pt\}$ the $G$-equivariant elements  
$\Q^H_X\in D^{b,G}\mh(X), \Q_X\in D^{b,G}_c(X;\Q)$ and $\OO_X\in \co^G(X)$.
\end{example}

Assume now that for all $X,X' \in var^{qp}/\C$ there is an exterior product
$$\boxtimes: A(X)\times A(X')\to A(X\times X'),$$
which is a bifunctor, ($\Q$-linear) additive (resp. exact) in each variable.
We require in addition to have K\"unneth isomorphisms
$${\it Kue}: f_*M\boxtimes f'_*M' \simeq (f\times f')_*(M\boxtimes M')\:,$$ 
functorial in $M\in ob(A(X))$ and $M'\in ob(A(X'))$, and satisfying certain compatibilities as spelled out in \cite{MS09}[Sect.4.3]. 
Then for $X$ a $G$-space and $X'$ a $G'$-space in our category $var^{qp}/\C$, with $G,G'$ finite groups, we get an induced equivariant exterior product
$$\boxtimes: A^G(X)\times A^{G'}(X')\to A^{G\times G'}(X\times X')$$
via
$$\psi_{(g,g')}:= {\it Kue} \circ (\psi_g,\psi_{g'}):  M\boxtimes M' \to (g_*M)\boxtimes (g'_*M')\simeq
(g,g')_*(M\boxtimes M').$$
And by the K\"unneth isomorphisms, this commutes with push-downs, i.e.,
\begin{equation}\label{Kue1}
 (f^G_*)\boxtimes (f'^{G'}_*) = (f,f')^{(G\times G')}_* \: .
\end{equation}

\begin{example}
Such an exterior product with K\"unneth isomorphisms is readily available in all  cases mentioned in 
Example \ref{pseudo} (as explained in more detail in \cite{MS09,MSS}).
Similarly, one has
\begin{equation}\label{Kue2}
 (f_G^*)\boxtimes ({f'_{G'}}^*) = (f,f')_{(G\times G')}^* \: 
\end{equation}
in all cases of Example \ref{pull}.
\end{example}

\subsection{Equivariant Grothendieck groups}
Since all categories (and functors) considered above are additive, we can make use of the corresponding Grothendieck group $\bar{K}_0(A(X))$, associated to the abelian monoid of isomorphism classes of objects with the direct sum. And similarly for the induced functors.

If, in addition, $A(X)$ is also abelian, the corresponding Grothendieck group 
$$K_0(A(X)):=\bar{K}_0(A(X))/\sim $$
is defined as the quotient of $\bar{K}_0(A(X))$ by the relations given by
$[M]=[M']+[M'']$ for any short exact sequence 
$$\begin{CD} 
0 @>>> M' @>>> M @>>> M'' @>>> 0   
  \end{CD}$$
in $A(X)$. This applies to the abelian categories $\mh(X)$,
$\co(X)$ and $Perv(X;\Q)$ from Example \ref{pseudo},
as well as to the corresponding equivariant categories (if the push down functors $f_*$ are exact).
These Grothendieck groups are functorial for exact functors of (equivariant) abelian categories,
e.g., push-forward for finite morphisms, exterior products, restriction to open subsets (resp.
restriction to subgroups and forgetful functors).\\

For a triangulated category $A(X)$, e.g., the derived categories $D^b\mh(X)$,
$D^b_{\rm coh}(X)$  or $D^b_c(X;\Q)$ from Example \ref{pseudo}, the correspnding 
Grothendieck group 
$$K_0(A(X)):=\bar{K}_0(A(X))/\sim $$
is defined as the quotient of $\bar{K}_0(A(X))$ by the relations given by
$[M]=[M']+[M'']$ for any distinguished triangle
$$\begin{CD} 
 M' @>>> M @>>> M'' @> [1] >>  M'[1]  
  \end{CD}$$
in $A(X)$. 
In this triangulated context we assume the push down functors $f_*$ to be exact, so that they induce similar group 
homomorphisms on the Grothendieck group level. 
But this notion can't be used directly for our (weak) equivariant derived 
categories $D^{b,G}\mh(X)$,
$D^{b,G}_{\rm coh}(X)$ and $D^{b,G}_c(X;\Q)$, since these are not triangulated.
Nevertheless, we can use instead a suitable notion of an {\em equivariant distinguished triangle} in $A^G(X)$, given by a triangle
$$\begin{CD} 
 M' @>\alpha >> M @> \beta >> M'' @> \gamma >>  M'[1]  ,
  \end{CD}$$
such that the underlying (non-equivariant) triangle is distinguished in the triangulated category $A(X)$, and all morphisms
$\alpha, \beta$ and $\gamma$ are $G$-equivariant (compare \cite{S}[Sect.3.1.1]).
The problem arises from the fact that not any equivariant morphisms can be extended
to such an equivariant distinguished triangle (because ``the'' cone is not unique).\\

In this paper we only need the corresponding Grothendieck group 
$$K_0(A^G(X)):=\bar{K}_0(A^G(X))/\sim \:,$$
defined as above by making  use of the equivariant distinguished triangles in $A^G(X)$.

Most of the standard calculus on Grothendieck groups of triangulated categories 
remains valid in our weaker context. For example, from an equivariant isomorphism $M\overset{\sim}{\to} M'$ one gets an equivariant distinguished triangle $M\overset{\sim}{\to} M'\to 0 \to M[1]$. Similarly, direct sums and shifts of equivariant distinguished triangles are of the same type. This implies:
\begin{enumerate}
\item[(a)] $[0]\in K_0(A^G(X))$ is the zero element of this group.
 \item[(b)] $[M]=[M']\in K_0(A^G(X))$, if there is an equivariant isomorphism 
$M\overset{\sim}{\to} M'$.
\item[(c)] $\left[M[1]\right]= - [M]$ and
 $\left[M\oplus M'\right]= [M]+ [M']$.
\end{enumerate}

\begin{example}\label{66}
Let $A$, $A'$ be covariant pseudofunctors taking values in triangulated categories. For $G$ a finite group acting on $X, X'$, consider an exact functor of triangulated categories
$F: A(X)\to A'(X')$ commuting with the $G$-action. This induces a 
$G$-equivariant functor $F^G: A^G(X)\to {A'}^G(X')$, with underlying functor
$F={\it For}(F^G)$, and a 
 group homomorphism $F^G: K_0(A^G(X))\to K_0({A'}^G(X'))$. This observation applies to the following examples used in this paper:
 \begin{enumerate}
  \item[(a)] For suitable $G$-morphisms like proper maps
(resp. smooth maps or open inclusions)
 $f: X\to Y$ in $var^{qp}/\C$, we get induced group homomorphisms like $f^G_*$ (resp. $f_G^*$), as well as exterior products.
\item[(b)] The functor ${\rm gr}^F_p DR^G:D^{b,G}\mh(X) \to D^{b,G}_{\rm coh}(X)$ used in our definition of the equivariant motivic Chern class transformation $\mc^G$.
 \end{enumerate}
 
\end{example}

Finally, the group isomorphisms
$$K_0(\mh^G(X)))\simeq  K_0(D^{b,G} \mh(X)) 
\quad \text{and} \quad
 K_0(D_{\rm coh}^{b,G}(X))\simeq
K_0({\rm Coh}^G(X)) $$
used in our construction of the equivariant motivic Chern class transformation $\mc^G$
follow from the following abstract result (compare \cite{S}[Lem.3.3.1]):

\begin{lemma}\label{last}
Let $T$ be covariant pseudofunctor taking values in triangulated categories, with  $G$ a finite group acting on $X$.
Assume $T(X)$ has a {\em bounded}  $t$-structure,
invariant under the action of the finite group $G$ (i.e., under $g_*$ for all $g\in G$). Let $Ab(X)$ be the heart of  this $t$-structure, with $h: T(X)\to Ab(X)$ the corresponding cohomology functor.
\begin{enumerate}
 \item[(a)] 
There is an induced equivariant functor
$h^G:  T^G(X)\to Ab^G(X)$, mapping equivariant distinguished triangles into equivariant
long exact cohomology sequences. So it induces a well-defined group homomorphism
$$\alpha: K_0\left( T^G(X)\right) \to K_0\left( Ab^G(X)\right);\: \ 
[M]\mapsto \sum (-1)^i\cdot \left[h^G\left( M[-i]\right)\right]\:$$
compatible with the forgetful functors ${\it For}$.
\item[(b)] Assume in addition that the $t$-structure on $T(X)$ is also {\em non-degenerate}.
Then $\alpha$ is an isomorphism of abelian groups 
$$K_0\left( T^G(X)\right) \simeq K_0\left( Ab^G(X)\right).$$
\end{enumerate}
\end{lemma}

\begin{proof}
Since the $t$-structure is invariant under the $G$-action, all push down functors $g_*$ ($g\in G$) commute with $h$.
So one gets  an induced equivariant functor
$h^G:  T^G(X)\to Ab^G(X)$, mapping equivariant distinguished triangles into equivariant
long exact cohomology sequences. This also defines the group homomorphism $\alpha$.
Note that the sum on the right hand side in the definition of $\alpha$ is finite, since the $t$-structure is {\em bounded}.

On the other hand, the heart $Ab(X)$ is an abelian {\em admissible} subcategory of $T(X)$, i.e.,
any short exact sequence in $Ab(X)$ underlies a {\em unique} distinguished triangle in
$T(X)$. So, if the short exact sequence is assumed to be $G$-equivariant, the same is true for the
corresponding distinguished triangle. Therefore, we also get a well-defined group homomorphism
$$\beta:  K_0\left( Ab^G(X)\right)\to K_0\left( T^G(X)\right); \: \ [M]\mapsto [M].$$

It remains to show that $\alpha$ and $\beta$ are inverse to each other in case the $t$-structure on $T(X)$ is also non-degenerate.
First note that $\alpha \circ \beta= id$, since for $M\in ob\left(Ab(X)\right)$ we get  
$h\left(M[-i]\right)=0$ for $i\neq 0$. Moreover, since the $t$-structure is {\em non-degenerate},
one obtains by definition that $h\left(M[-i]\right)=0$ for all $i$ implies $M=0$. Next we show as in \cite{S}[Lem.3.3.1] 
that $\beta\circ \alpha([M])=[M]$ by using induction over 
$$l(M):=\sharp\{i\in \Z\vert \; h\left(M[-i]\right) \neq 0 \;\}.$$ 
For the induction step we use the distinguished triangle
$$\begin{CD} 
 h\left(M[-i]\right)[-i] @>>> \tau^{\geq i}M @>>>  \tau^{> i}M @>[1]>> h\left(M[-i]\right)[-i+1] 
  \end{CD}$$
for $\tau^{\geq i}$ resp. $\tau^{> i}$  the corresponding truncation functor of the $t$-structure. By the definition of a $t$-structure, the degree one map is {\em unique}.
So a $G$-action on $M$ induces a $G$-action on this distinguished triangle.
This completes the induction step.
\end{proof}

Note that in our examples, $D^{b}\mh(X)$ resp. $D_{\rm coh}^{b}(X)$, the standard $t$-structure is
bounded and non-degenerate, with heart $\mh(X)$ resp. $\co(X)$. Moreover, this $t$-structure is invariant
under the action of $G$, since  for an isomorphism $g$ of $X$ the push-forward $g_*$ is $t$-exact in each of these cases.
Similar considerations apply in the case of the (self-dual) perverse $t$-structure on $D^b_c(X;\Q)$ (for the middle perversity), with heart $Perv(X;\Q)$.\\

As a byproduct of Lemma \ref{last}, one has that for a $G$-equivariant morphism $f: X\to X'$ in $var^{qp}/\C$ and 
$\MC\in D^{b,G}\mh(X)$, all  $R^if_*(\MC):=h^i(f_*\MC)\in \mh(X')$ get an induced $G$-action ($i\in \Z$), making the following diagram commutative:

\begin{equation}\begin{CD}
K_0(D^{b,G}\mh(X)) @> \alpha > \sim > K_0(\mh^G(X)) \\
@V f^G_* VV @V f^G_*:= V \sum_{i}\; (-1)^i\cdot [ R^if_*(-) ] V \\
K_0(D^{b,G}\mh(X')) @> \alpha > \sim > K_0(\mh^G(X')) \:.
\end{CD}\end{equation}

So for the right vertical group homomorphism $ f^G_*: K_0(\mh^G(X))\to K_0(\mh^G(X'))$
one can avoid the use of (weakly) equivariant derived categories  (this is the approach taken in \cite{T}).
But then one needs to prove the
functoriality of this definition, which is easier to get for the left vertical group homomorphism
$$f^G_*: K_0(D^{b,G}\mh(X))\to K_0(D^{b,G}\mh(X'))\:.$$ For the functoriality of the right hand side, one can use a 
corresponding {\em equivariant spectral sequence} $R^if_*\circ R^jf'_*\Rightarrow R^{i+j}(f\circ f')_*$ in  $\mh^G(-)$,
which one directly gets from Lemma \ref{last} and the corresponding $t$-structures.
Similar considerations apply for the other functors  $f_!, f^*$ and $f^!$ on the level of $K_0(D^{b,G}\mh(-))$,
whereas $\boxtimes$ is already exact in both variables. And the corresponding equivariant group homomorphisms on the level of $K_0(D^{b,G}\mh(-))$
have the same  calculus as
the underlying non-equivariant group homomorphisms, e.g., properties like {\em functoriality, projection formula, smooth base-change} and 
(\ref{Kue1}), (\ref{Kue2}).\\

Let us conclude the paper by explaining  the fact that the transformation
$$\chi^G_{\rm Hdg}: K^G_0 (var/X) \to K_0 (\mh^G(X))\simeq K_0 (D^{b,G}\mh(X)),$$ $$[f:Y \to X] \mapsto [f_!\Q^H_Y]$$from the relative Grothendieck group of $G$-equivariant quasi-projective varieties over $X$  to the Grothendieck group of $G$-equivariant mixed Hodge modules is well-defined, i.e., it satisfies the ``scissor" relation (compare with \cite{BSY,Sch3} for the non-equivariant counterpart). Let $i:Z \to Y$ be the closed inclusion of a $G$-invariant subvariety, with open complement $j:\U:=Y \setminus Z \to Y$. Then there is a distinguished triangle in $D^b\mh(Y)$ (see \cite{Sa}[(4.4.1)]:
$$\begin{CD} 
j_!\Q^H_{\U} @>{\rm ad}_j>> \Q^H_Y @>{\rm ad}_i>> i_!\Q^H_Z @>[1]>> , \end{CD}$$
with ${\rm ad}_j$ and ${\rm ad}_i$ the adjunction morphisms for the inclusions $j$ and $i$. By our definition,  ${\rm ad}_j$ and ${\rm ad}_i$ are $G$-equivariant morphisms. Similarly, the shift morphism $[1]$ is $G$-equivariant because  the degree one map $ i_!\Q^H_Z \to  j_!\Q^H_{\U}[1]$ is {\em unique} by 
$$Hom_{D^b{\rm MHM}(Y)}(i_!(-),  j_!(-))=0.$$ The last assertion follows by adjunction, since $j^*i_!=0$ (as can be checked on the underlying constructible sheaf complexes). So, we get:
$$[ \Q^H_Y]=[j_!\Q^H_{\U}]+[i_!\Q^H_Z] \in K_0 (D^{b,G}\mh(Y))\simeq K_0 (\mh^G(Y)).$$ Finally, the scissor relation:
$$[ f_!\Q^H_Y]=[(f\circ j)_!\Q^H_{\U}]+[(f\circ i)_!\Q^H_Z] \in K_0 (D^{b,G}\mh(X))\simeq K_0 (\mh^G(X))$$ follows by applying the functor $f_!$.
Note that $\chi^G_{\rm Hdg}$ commutes with push-downs $g_!$, exterior product and  pull-backs (e.g., open restrictions), as well as forgetful functors for subgroups of $G$.

\end{document}